\documentclass[a4paper]{article}
\usepackage[margin=3cm]{geometry}

\usepackage[utf8]{inputenc}
\usepackage[T1]{fontenc}
\usepackage{lmodern,amsmath,amssymb,amsfonts,dsfont,mathrsfs}
\usepackage{natbib}
\usepackage{xcolor}
\usepackage{tikz}
\usepackage{bm}
\usepackage[colorlinks,linkcolor=blue!70!black,citecolor=magenta!90!black]{hyperref}
\usepackage{mathtools}
\usepackage{doi}
\RequirePackage[amsthm,amsmath,thmmarks,hyperref,framed]{ntheorem}

\usepackage[shortlabels]{enumitem}
\newlist{mathlist}{enumerate}{5}
\setlist[mathlist]{label={\textup{(\roman*)}}}
\theoremstyle{plain}
\theoremseparator{.}
\newtheorem{theorem}{Theorem}
\newtheorem{prop}[theorem]{Proposition}
\newtheorem{lem}[theorem]{Lemma}
\newtheorem{cor}[theorem]{Corollary}

\theoremstyle{definition}
\newtheorem{assumption}{Assumption}

\newtheorem{remark}[theorem]{Remark}
\makeatletter
\newcommand*{\transpose}{%
  {\mathpalette\@transpose{}}%
}
\newcommand*{\@transpose}[2]{%
  \raisebox{\depth}{$\m@th#1\intercal$}%
}
\makeatother

\usepackage{setspace}

\renewcommand{\epsilon}{\varepsilon}

\renewcommand{\rho}{\varrho}
\renewcommand{\phi}{\varphi}

\renewcommand{\geq}{\geqslant}
\renewcommand{\leq}{\leqslant}

\newcommand{\norm}[2][]{#1\lVert #2 #1\rVert}
\newcommand{\abs}[2][]{#1\lvert #2 #1\rvert}
\newcommand{\floor}[2][]{#1\lfloor #2 #1\rfloor}

\title{Sharp $L \log L$ condition for supercritical Galton--Watson processes with countable types}

\usepackage{authblk}

\author[1,2]{Mathilde \textsc{André}} 
\author[3]{Jean-Jil \textsc{Duchamps}}

\affil[1]{Institut de Biologie de l'ENS (IBENS), École Normale Supérieure, PSL Université, CNRS UMR 8197,
INSERM U1024, Paris, France}
\affil[2]{Faculty of Mathematics, University of Vienna, Oskar-Morgenstern-Platz 1, 1090 Wien, Austria}
\affil[3]{Marie et Louis Pasteur University, CNRS, LmB (UMR 6623), F-25000 Besançon, France}

\begin{document}

\maketitle

\begin{abstract}
  We investigate Kesten--Stigum-like results for multi-type
  Galton--Watson processes with a countable number of types in a
  general setting, allowing us in particular to consider processes
  with an infinite total population at each generation. Specifically,
  a sharp $L\log L$ condition is found under the only assumption that
  the mean reproduction matrix is positive recurrent in the sense
  of~\cite{vere-jones_ergodic_1967}. The type distribution is shown to
  always converge in probability in the recurrent case, and under
  conditions covering many cases it is shown to converge almost
  surely. We finally apply these results to the study of a
    directed network built from countably many configuration models as
    a model of epidemics.
\end{abstract}

\setcounter{tocdepth}{1}
\tableofcontents
\defcitealias{andre_duchamps_directed_2025}{André and Duchamps (in prep.)}

\section{Introduction and main results} \label{sec:intro}

The large-time behavior of Galton--Watson processes, both in the
single-type and multi-type settings, has been thoroughly explored in
various
works~\citep[\textit{e.g.}][]{harris_theory_1964,athreya_branching_1972}.
However, more general settings remain less investigated. The
Galton--Watson process with countably many types arises naturally when
studying discretized structured population, such as aging processes or
spatially structured populations (\textit{e.g.}
\citealt{bansaye_queueing_2018,ligonniere_kesten_2024,braunsteins_pathwise_2019}).
Notably, although some Kesten--Stigum limit theorems have been
established for such processes, these results often assume additional
conditions such as compactness of the mean reproduction operator or
some kind of uniform control of the types. These assumptions
  are natural to deal with general type spaces, but they may seem more
  artificial for countable type space. This was apparent when studying
  an age-structured branching process arising from a stochastic model
  of epidemics, which we present in the last part of the article (see
  Section~\ref{sec:applications}). We thus explored weak sufficient
  conditions to obtain a crucial result for this application: the
  general convergence of the type distribution of countable-type
  branching processes in probability stated in
  Theorem~\ref{thm:type_convergence}~(ii). Along the way, we found
  some improvements over the long-existing Kesten--Stigum theory in
  the supercritical countable-type case
  \citep{moy_extensions_1967,asmussen_strong_1976}, with assumptions
  that are minimal regarding the asymptotic behavior of the iterated
  mean reproduction matrix (\textit{i.e.} we only assume a Perron--Frobenius
  behavior). We also related our findings with the phenomenon of local
  extinction, which concerns a number of recent works
  \citep[\emph{e.g.}][]{braunsteins_pathwise_2019,ligonniere_kesten_2024}.
 To the best of our knowledge, no sharp Kesten--Stigum-like result
had been provided yet for the objects we consider, namely the general
supercritical countable-type recurrent Galton--Watson processes. Our
aim is to provide one such result, using as our main ingredient the
spinal construction adapted from the seminal work of
\citet{lyons_pemantle_peres_1995}. This gives the theoretical results
an intrinsic interest.

This article is organized as follows. In the remainder of
  Section~\ref{sec:intro}, we introduce the model and notation and
  state our main results (Theorem~\ref{thm:ks_condition} and
  Theorem~\ref{thm:type_convergence}). Section~\ref{sec:examples} is
  devoted to a simple way of building examples of countable-type
  branching processes, allowing us to provide examples and
  counter-examples to some of our statements. In
  Section~\ref{sec:proof_ks_results}
  and~\ref{sec:proof_type_convergence} we present the proofs of the
  main results. Finally we exploit our results in
  Section~\ref{sec:applications} in the study of the local weak limit
  of a countable layering of configuration models as a general model
  of epidemics.

\subsection*{Notation}
In the following, we consider a countable set $\mathcal{X}$. We will write $\mathbb{R}^{\mathcal{X}}$ and $\mathbb{N}^{\mathcal{X}}$ for the column vectors indexed by $\mathcal{X}$ with entries in respectively $\mathbb{R}$ and $\mathbb{N}$ (in this article we denote by $\mathbb{N}$ the set of \emph{nonnegative} integers).
For $u, v\in\mathbb{R}^{\mathcal{X}}$ we denote by $u^\transpose$ the transpose of $u$ such that the inner product (when it makes sense) is $\langle u,v\rangle \coloneqq u^\transpose v = \sum_{x\in\mathcal{X}}u_x v_x\in \mathbb{R}$ and $vu^\transpose \in\mathbb{R}^{\mathcal{X}\times \mathcal{X}}$ is the matrix with entries $(v_x u_y)_{x,y\in\mathcal{X}}$.
Note that we will only deal with vectors and matrices with nonnegative entries, so that matrix multiplication is always well-defined and associative.
If $\mathcal{X}'\subset \mathcal{X}$ we write $v|_{\mathcal{X}'}$ for the vector $(v_x\mathds{1}_{x\in \mathcal{X}'})_x$, so that in particular $\langle u, v|_{\mathcal{X}'}\rangle \coloneqq\sum_{x\in\mathcal{X}'}u_xv_x$.
We will also use the notation $e_y=(\mathds{1}_{x=y})_{x}\in \mathbb{N}^{\mathcal{X}}$.

Throughout this work we will use the notion of biased random variables, always in the same usual sense which we recall below. 
Let $A$ be any random variable taking values in a measurable space
$(E, \mathcal{A})$ and $B$ a nonnegative random variable with positive
and finite expectation. Then, the sentence ``$A$ \emph{biased by} $B$'', refers to a random variable $A'$ with values in $E$ such that the distribution of $A'$ is characterized by
\begin{equation}\label{eq:biased_rv}
\mathbb{E}\big[F(A')\big]\ =\ \frac{\mathbb{E}[BF(A)]}{\mathbb{E}[B]},
\end{equation}
for all nonnegative measurable functionals $F:E \to\mathbb{R}_+$.
Note that often, we will take use $E= \mathbb{R}^{\mathcal{X}}$,
equipped with the product $\sigma$-algebra.

\subsection*{The model}
We consider a multi-type Galton--Watson process, with a countable set
of types $\mathcal{X}$. This framework naturally models the
  dynamics of a structured population of particles.

The model is defined from a fixed family of distributions
  $(\mathcal{L}_x)_{x\in \mathcal{X}}$ on $\mathbb{N}^{\mathcal{X}}$.
  For each $x,y \in \mathcal{X}$, for a random variable
  $L^{(x)}\sim \mathcal{L}_x$, we write $L^{(x)}_y$ for its
  $y$-marginal, and assume that the first moment
  $M_{xy}\coloneqq\mathbb{E}[L^{(x)}_y]$ is finite. We define the
  nonnegative matrix $M = (M_{xy})_{x,y\in \mathcal{X}}$ and further
  assume that $M^n$ has only finite coordinates for all $n\geq 1$. Let
  us fix some $Z(0)\in \mathbb{N}^{\mathcal{X}}$ with a finite number
  of nonzero entries. The Galton--Watson process with offspring
  distributions $(\mathcal{L}_x)_{x\in \mathcal{X}}$ and started from
  $Z(0)$ is defined as a sequence of $\mathbb{N}^{\mathcal{X}}$-valued
  random variables $(Z(n))_{n\geq 0}$ satisfying the induction:
  \[
    Z(n+1)\ =\ \sum_{x\in \mathcal{X}} \sum_{i=1}^{Z_x(n)} L^{(x,i,n)},
    \qquad n\geq 0,
  \]
  where $(L^{(x,i,n)})_{x\in \mathcal{X}, i\geq 1, n\geq 0}$ is a
  collection of independent random variables, with
  $L^{(x,i,n)}\sim \mathcal{L}_x$.

The random variable $Z_y(n)$ models the number of type-$y$
  individuals in the $n$th generation of a population in which
  individuals of any given type $x\in \mathcal{X}$ reproduce
  independently of the others with a progeny distributed as the random
  variable $L^{(x)}$. We point out the fact that the progeny of an
individual is almost surely finite coordinate-wise, however their
total number of children need not be. Therefore, the total number of
individuals in the population at generation $n$, given by
$\sum_{x\in\mathcal{X}}Z_x(n)$, is allowed to be infinite. Note
  that the assumption that $M^n < \infty$ (any equality or inequality
  we write concerning matrices or vectors means that it holds
  coordinate-wise) for all $n\geq 0$ implies that for any
  $x,y\in \mathcal{X}$, starting from $Z(0)=e_x$, we have
  $\mathbb{E}[Z_y(n)]=M^n_{xy} < \infty$.

We summarize here some known definitions and results on the extension of Perron--Frobenius theory to countable nonnegative matrices, for which we refer to~\citet[Chapter~7, Section~1]{kitchens_symbolic_1998}; see also~\citet[Chapter~6]{seneta_non-negative_1981} or~\citet{vere-jones_ergodic_1967}.
The matrix $M$ is said to be \textit{irreducible} if the adjacency graph of $M$ (\textit{i.e.} the graph with vertex set $\mathcal{X}$ and where $x\to y$ is a directed edge iff $M_{xy}>0$) is strongly connected.
It is said to be \textit{aperiodic} if for all $x\in\mathcal{X},\
\gcd\{n :  M^{(n)}_{xx}>0\}=1$.
If $M$ is irreducible the period is common to all indices and ``for all'' can be replaced by ``for some''.
Let $M$ be irreducible and aperiodic.
Its Perron value, defined as 
\begin{equation}\label{eq:rho}
  \varrho\ \coloneqq\ \lim_{n\to\infty}(M^n_{xx})^{1/n}\ \in\ (0,\infty], 
\end{equation}
  is independent of $x\in\mathcal{X}$ and $\varrho^{-1}$ is the radius of convergence of the power series $f^{(x,y)}:s\mapsto\sum_{n\in \mathbb{N}} M^{n}_{xy}s^n$, for all pairs $(x,y)$.
Then, $M$ is said to be \textit{recurrent} if $f^{(x,x)}(\varrho^{-1})=\infty$ for some $x$ (hence, for all $x$), otherwise \textit{transient}.

In the finite type setting and under assumptions of irreducibility and aperiodicity, the Perron--Frobenius theorem gives the mean behavior of the process, whose exponential growth is driven by the greatest eigenvalue of the mean matrix.
When it comes to infinite non-negative matrices, the following proposition~\citep[Theorem~7.1.3]{kitchens_symbolic_1998} gives an analogous result.

\begin{prop}[Extension of Perron--Frobenius theory to infinite matrices]\label{prop:perron_frobenius}
  Assume $M$ is irreducible, aperiodic and recurrent, with a finite Perron value $\varrho$. 
  Then, $\varrho$ is the largest eigenvalue of $M$ and is simple. Moreover, there exists associated left and right eigenvectors $\tilde{h}$ and $h$ in $\mathbb{R}^{\mathcal{X}}$, which are positive and unique up to multiplicative constants.

  If $\langle \tilde{h}, h\rangle<\infty$, then  $M$ is said to be \emph{positive recurrent} and normalizing the eigenvectors such that $\langle \tilde{h}, h\rangle=1$, we have that for all $x,y\in \mathcal{X}$,
  \begin{equation*}
    \lim_{n\to\infty} \varrho^{-n}M^{n}_{xy}\ =\ \big(h\tilde{h}^\transpose\big)_{xy}\ =\  h_x\tilde{h}_y.
   \end{equation*}   
  Otherwise, if $\langle \tilde{h}, h\rangle=\infty$, $M$ is said to be \emph{null recurrent} and we have $\lim_{n\to\infty} \varrho^{-n}M^{n} = 0$ coordinate-wise.
\end{prop}
From now on, we work under the following assumption.
\begin{assumption}\label{ass:recurrence}
  The mean offspring matrix $M$ is irreducible, aperiodic and recurrent, with a Perron value $1<\varrho<\infty$. In the case where $M$ is positive recurrent, we will always assume that the normalization $\langle \tilde{h}, h\rangle=1$ holds.
\end{assumption}
%

\subsection*{Related works}
Kesten--Stigum-like results typically consist in finding conditions for the non-degeneracy of the a.s.\ limit $W$ of the so-called additive martingale of the process (defined as $W_n\coloneqq \varrho^{-n} \langle Z(n), h\rangle$), relating it to the extinction event $\mathrm{Ext}\coloneqq\{\exists n \geq 0 : Z(n)=0\}$ of the process, and studying the asymptotic type-distribution (\textit{i.e.} the behavior of $\varrho^{-n}Z(n)$).
For multi-type branching processes with a \textit{finite} set of types, under Assumption~\ref{ass:recurrence} and with $\langle  \tilde{h}, h \rangle=\langle \tilde{h} , 1\rangle =1$, it is well-known from the seminal work of~\cite{kesten_limit_1966} that 
\begin{equation*}
   \mathbb{E}\bigg[\sum_{x,y\in\mathcal{X}}L^{(x)}_y\log_+\big(L^{(x)}_y\big)\bigg]<\infty \ \Longleftrightarrow\  \begin{cases}
    \mathbb{E}[W]\ =\ \langle Z(0), h\rangle, \\ \mathbb{P}(W>0)\ >\ 0, \\ \{W=0\}\ =\ \mathrm{Ext},\  \text{up to a null-measure set},
   \end{cases}
\end{equation*}
the three items on the right-hand side being equivalent.
Moreover, they showed that the normalized distribution of the types in the population converges \textit{almost surely} to $W\tilde{h}$ as $n\to\infty$.
A short proof of this result is the main result of~\cite{kurtz_conceptual_1997}, building on the spinal change of measure developed by the same authors for single type branching processes~\citep{lyons_pemantle_peres_1995}.

When it comes to broader sets of types, several studies investigated the limit behavior of the large time type-population vector in the supercritical Galton--Watson process.
Limit theorems for the countable-type setting already appear in the seminal works of~\citet[Chapter~III]{harris_theory_1964} and~\citet{moy_extensions_1967}.
Both derive extinction criterion and almost sure convergence of $\varrho^{-n}Z(n)$ under strong additional assumptions.
\citet{moy_extensions_1967} works with a second moment assumption, and~\citet[Chapter~III]{harris_theory_1964} assumes the coefficients of $M$ are uniformly bounded.
Some more recent
works~\citep{loic_de_raphelis_scaling_2017,topchii_critical_2020}
focus on the critical case. They derived Yaglom-type results, that is,
convergence in distribution of the type vector conditional upon
non-extinction.

Several works have focused on $L\log L$ criteria, in the countable
setting or in more general multi-type settings, however, these works
generally require some kind of compactness assumption or uniform bound
on moments of the offspring distribution. \citet{asmussen_strong_1976}
give a necessary and sufficient $L\log L$ condition and almost sure
convergence of the types assuming a uniformly bounded mean number of
offspring over the type space (which is not assumed to be
  countable), and a uniform convergence assumption of the
renormalized first moment semigroup. Very recently,~\citet{BBC25}
manage to relax these uniformity assumptions under a slightly stronger
$L \log L$ condition. \citet{kesten_supercritical_1989} works under
polynomial bounds for the fitness (that is, the contribution of the
progeny of an individual to the demography of the population in the
long run) of an individual outside a set $A_t$ that does not grow too
quickly. \citet{biggins_multi-type_1999} tackle $L^p$ convergence of
the type distribution in a random environment, under ergodicity
conditions and uniform integrability conditions.
\citet{athreya_change_2000} and \citet{kyprianou_measure_2004} give
criteria for the uniform integrability of the additive martingale in
general type spaces, involving uniform $L \log L$ moment assumptions.
More recently~\cite{ligonniere_kesten_2024}, in the setting of varying
environment and countable types, also obtains sufficient conditions
for convergence of the type distribution, under some uniform control
of the second moments of the offspring vectors. In particular, in all
of these works the total population size is almost surely finite at
all times, but we needed to get rid of this assumption in our upcoming
work, and our present results show this is indeed perfectly viable.
Note that infinite populations in finite time appear quite naturally
in the recent theory of (growth-)fragmentation processes: see in
particular \citet[Theorem~1.1]{bertoin_strong_2020}, where the authors
study growth-fragmentation branching processes and derive ergodicity
results similar to our Theorem~\ref{thm:type_convergence}, in a quite
different setting of continuous time and type space.

Our results could find application in other contexts, such as seed bank models or delayed branching processes where countable sets of types typically appear.
Indeed, branching processes with potentially unbounded dormant populations are used to model long-term genetic storage \citep{Casanova_Kurt_Spanò_2013}. 

In a more specific setting, discretized Crump--Mode--Jagers processes studied in~\cite{jagers_general_2008} can also be interpreted as branching processes with a countable set of types.
For these specific processes independence between the progeny produced by a same individual at distinct time allows obtaining Kesten--Stigum results and limit theorems under minimal assumptions, as is well documented in the work of~\cite{nerman_convergence_1981}.

\subsection*{Main results}
We seek to generalize Kesten--Stigum uniform integrability and ergodicity results to the countable type setting under minimal assumptions (namely Assumption~\ref{ass:recurrence}).
We start from a population satisfying $\langle Z(0), h\rangle <\infty$.
Because $h$ is a right eigenvector for $M$, for every $n\geq 0$ we have 
\[
\mathbb{E}\big[\langle Z(n), h\rangle \big]\ =\ \varrho^n \langle Z(0), h\rangle\ <\ \infty,
\]
 and in particular $\langle Z(n), h\rangle<\infty$ holds almost surely.

This way, the population vector at generation $n$ can be seen as a random element of the Banach space $\mathbb{B}_h=\{v\in \mathbb{R}^{\mathcal{X}} : \norm{v}_{h} \coloneqq \sum_x \abs{v_x}h_x < \infty\}$. In the case where $M$ is positive recurrent, $\tilde{h}$ is a unit vector in $\mathbb{B}_h$. In the same spirit as the additive martingale in finite multi-type branching theory, it is not hard to see that for all $n\in \mathbb{N}$, the process  
\begin{equation*}
W_n\ \coloneqq\ \varrho^{-n}\langle Z(n), h\rangle
\end{equation*}
defines a (scalar) nonnegative martingale. We denote by $W$ its almost sure limit as $n\to\infty$. 

We will see that in the general setting of recurrence for $M$, a sufficient condition for $\mathbb{E}[W] = \langle Z(0),h\rangle$ -- or equivalently, for the uniform integrability of $(W_n)_n$ -- is
\begin{equation}\label{eq:simple_condition}\tag{A}
  \sum_{x\in \mathcal{X}} \tilde{h}_x \mathbb{E}\Big[\big\langle L^{(x)},h\big\rangle\log_+\big\langle L^{(x)},h\big\rangle\Big]\ <\ \infty.
\end{equation}
This relates to previous works, in particular to~\citet{asmussen_strong_1976} where the analog of this condition is necessary and sufficient to have $\mathbb{E}[W] = \langle Z(0),h\rangle$ in their setting.
Let us also mention \citet{athreya_change_2000}, whose type space is only assumed to be a measurable space (not a topological space).
Translating to our notation, his sufficient (but not necessary) condition for the uniform integrability of the additive martingale reads:
\[
  \int_0^{\infty} \sup_{x\in \mathcal{X}}\frac{1}{h_x} \mathbb{E}(\langle L^{(x)}, h\rangle \mathds{1}_{\langle L^{(x)}, h\rangle > \mathrm{e}^t}) \,dt\ <\ \infty.
\]
In our positive recurrent setting, note that this is stronger than~\eqref{eq:simple_condition}.
Indeed, using the fact that $(\tilde{h}_xh_x)_{x\in \mathcal{X}}$ sums up to $1$, we compute:
\begin{align*}
  \int_0^{\infty} \sup_{x\in \mathcal{X}}\frac{1}{h_x} \mathbb{E}\big[\langle L^{(x)}, h\rangle \mathds{1}_{\langle L^{(x)}, h\rangle > \mathrm{e}^t}\big] \,dt \ &\geq\ \int_0^{\infty} \sum_{x\in \mathcal{X}}\frac{\tilde{h}_xh_x}{h_x} \mathbb{E}\big[\langle L^{(x)}, h\rangle \mathds{1}_{\langle L^{(x)}, h\rangle > \mathrm{e}^t}\big] \,dt\\
  & =\ \sum_{x\in \mathcal{X}}\tilde{h}_x \mathbb{E}\Big[\langle L^{(x)}, h\rangle \log_+ \langle L^{(x)}, h\rangle \Big].
\end{align*}
Note that~\eqref{eq:simple_condition} is \emph{not necessary} in our general setting, even when $M$ is further assumed to be positive recurrent -- in this case we can express a necessary and sufficient condition that is less easily checked.

In order to state the latter condition, we need some more definitions.
Let us define the \emph{spinal chain}, that is a Markov chain $(X_n)_{n\geq 0}$ on $\mathcal{X}$ with transition probabilities
\begin{equation}\label{eq:spine}
  p_{xy}\ =\ \frac{M_{xy}h_y}{\varrho h_x}.
\end{equation}
Note that $X$ is a recurrent Markov chain when $M$ is recurrent. Denoting by $\pi$ the vector with coordinates
\begin{equation}\label{eq:stationary}
\pi_x\ \coloneqq\ \tilde{h}_xh_x,\quad x\in \mathcal{X},
\end{equation}
we have that $\pi$ is a stationary distribution for $X$, so that positive recurrence of $M$ is equivalent to positive recurrence of $X$.

Let us define the $\sigma$-algebra $\mathcal{G}=\sigma(X_n, n\geq 0)$ and on the same probability space, introduce the sequence of nonnegative random variables $(Y_n)_{n\geq 1}$ such that:
\begin{itemize}
  \item The $(Y_n)$ are independent conditional on $\mathcal{G}$.
  \item The distribution of $Y_n$ conditional on $\mathcal{G}$ and on
    $(X_{n-1}, X_{n})=(x,y)$ is the distribution of
    $\langle L^{(x)}, h\rangle$ biased by $L^{(x)}_y$ in the sense
    of~\eqref{eq:biased_rv}, taking $E=\mathbb{R}$. That is,
  \begin{equation}\label{eq:Yn_biased}
    \mathbb{E}\Big[F(Y_n) \mid \mathcal{G}\Big]\mathds{1}_{X_{n-1}=x, X_{n}=y}\ =\ \frac{1}{M_{xy}} \mathbb{E}\Big[L^{(x)}_y F\big(\langle L^{(x)},h\rangle\big)\Big] \mathds{1}_{X_{n-1}=x, X_{n}=y},
  \end{equation}
  for any nonnegative measurable function $F:\mathbb{R}\to \mathbb{R}$.
\end{itemize}
We will write $\mathbb{E}_x$ to denote expectation conditional on $\{X_0=x\}$.
The random sequence $(X_n,Y_n)$ allows us to express the sharp $L \log L$ condition in the positive recurrent case, namely:

\begin{theorem}[Uniform integrability of the additive martingale]\label{thm:ks_condition}
  Fix any state $x$, and denote by 
  \[T_1\ \coloneqq\ \inf\big\{k\geq 1,\ X_k=x \big\}
  \] 
   the first return time to $x$ for the spinal Markov chain $X$. Assume that the initial population satisfies $\langle Z(0), h\rangle<\infty$ and consider the condition:
  \begin{equation}\label{eq:nsc}\tag{B}
    \mathbb{E}_x\bigg[\log_+\Big(\sum_{k=1}^{T_1}\varrho^{-k} Y_k\Big)\bigg]\ <\ \infty.
  \end{equation}
If $M$ is positive recurrent, then~\eqref{eq:nsc} is equivalent to any of the following (equivalent) points:
  \begin{mathlist}
    \item\label{eq:point_1_ext} $\mathbb{E}[W]=\langle Z(0), h\rangle$;
    \item\label{eq:point_2_ext} $\mathbb{P}(W>0)>0$;
    \item\label{eq:point_3_ext} $\{W=0\}=\mathrm{LocExt}$ up to a null-measure set,
    where we denote by \[
    \mathrm{LocExt}\ =\ \big\{\forall x\in \mathcal{X}, \lim_{n\to\infty} Z_x(n)=0\big\}
    \] the event of local extinction.
  \end{mathlist}
  If $M$ is recurrent and Condition~\eqref{eq:simple_condition} or Condition~\eqref{eq:nsc} is satisfied, then the points~\ref{eq:point_1_ext},~\ref{eq:point_2_ext} and~\ref{eq:point_3_ext} also hold.
\end{theorem}
\begin{remark}
Theorem~\ref{thm:ks_condition} states that in the positive recurrent setting, Condition~\eqref{eq:simple_condition} is sufficient to obtain uniform integrability of the sequence $(W_n)_n$, however unlike~\eqref{eq:nsc}, Condition~\eqref{eq:simple_condition} is not necessary (see Section~\ref{sec:examples} for an example).
If $M$ is null recurrent, then both conditions are sufficient.
\end{remark}
\begin{remark}[Extinction of the branching process]\label{rem:extinction}
As highlighted before, the Kesten--Stigum theorem on branching
processes with a finite set of types includes the fact that $\mathbb{P}(W>0)>0$ is equivalent to $\mathrm{Ext}=\{W=0\}$ up to a null-measure set, where $\mathrm{Ext}$ is the event of extinction. This is usually checked through the immediate inclusion $\{W=0\}\subset\mathrm{Ext}$ and the equality $\mathbb{P}(W=0)\in\{\mathbb{P}(\mathrm{Ext}),1\}$, which is derived using a fixed point argument on the generating function of the offspring distribution~\citep{harris_theory_1964}.

However, when the set of types is infinite the question is much more involved. As a matter of fact, the number of fixed points of the offspring generating function is not 2 but is at least countably infinite, as shown in~\cite{tan_characterization_2023}.
In fact, every type could become extinct eventually while the whole population does not, and may even grow exponentially. This phenomenon of local extinction has been studied in many recent works~\citep{braunsteins_pathwise_2019,ligonniere_kesten_2024, kyprianou_extinction_2018}. Examples of such processes can be found in~\cite{braunsteins_pathwise_2019}, and we give in Section~\ref{sec:examples} an instance where even under positive recurrence of the mean matrix we have $\mathbb{P}(W>0)>0$ yet $\mathbb{P}\big(\mathrm{Ext}^c\cap\{W=0\}\big)>0$. 
Nevertheless, in our countable-type setting there is a dichotomy between the events $\{W>0\}$ and local extinction of the process.
\end{remark}

\begin{remark}\label{rem:simple_implies_nsc}
  Using the fact that $\log_+(\sum_i a_i) \leq \sum_i \log(1+a_i)$ for any sequence of nonnegative numbers $(a_i)_i$, we get
  \[
    \mathbb{E}_x\bigg[\log_+\Big(\sum_{k=1}^{T_1}\varrho^{-k} Y_k\Big)\bigg]\ \leq\ \mathbb{E}_x\bigg[\log_+\Big(\sum_{k=1}^{T_1} Y_k\Big)\bigg]\ \leq\ \mathbb{E}_x\bigg[\sum_{k=1}^{T_1}\log(1+Y_k)\bigg],
  \]
  and recalling from~\eqref{eq:stationary} the stationary distribution $\pi$ of $X$, the last term above can be rewritten
  \begin{align*}
    \frac{1}{\pi(x)} \sum_{y\in \mathcal{X}}\pi(y)\mathbb{E}_y\big[\log(1+Y_1)\big]
   \ &=\ \frac{1}{\pi(x)} \sum_{y,z\in \mathcal{X}}\pi(y) \frac{M_{yz}h_z}{\varrho h_y} \cdot \frac{1}{M_{yz}}\mathbb{E}\big[L^{(y)}_z \log\big(1+\langle L^{(y)},h \rangle\big)\Big]\\
   &=\ \frac{1}{\pi(x)} \sum_{y\in \mathcal{X}}\tilde{h}_y \mathbb{E}\Big[\langle L^{(y)},h \rangle \log\big(1+\langle L^{(y)},h \rangle\big)\Big],
  \end{align*}
  which, in the positive recurrent case, is finite only if~\eqref{eq:simple_condition} holds, showing that~\eqref{eq:simple_condition} is stronger than~\eqref{eq:nsc}.
  In the null recurrent case it is not clear with this argument, but we will prove directly that~\eqref{eq:simple_condition} is sufficient for the uniform integrability of $(W_n)$.
\end{remark}

We now turn to the large population asymptotic of the type distribution $Z(n)/\varrho^n$. We recall that we work on the Banach space
\[
  \mathbb{B}_h=\Big\{v\in \mathbb{R}^{\mathcal{X}} : \norm{v}_{h}
    \coloneqq \sum_x \abs{v_x}h_x < \infty \Big\},
\]
and denote by $L^1(\Omega,\mathcal{A},\mathbb{P},\mathbb{B}_h)$ the subspace of random variables $V:\Omega\to \mathbb{B}_h$ such that $\mathbb{E}[\norm{V}_h]<\infty$. For reasons of brevity, we will use the notation of~\citet[Chapter~1]{pisier_martingales_2016} and write  $L^1(\mathbb{B}_h)$ instead of $L^1(\Omega,\mathcal{A},\mathbb{P},\mathbb{B}_h)$.

\begin{theorem}[Convergence of the type distribution]\label{thm:type_convergence}\hfill
  \begin{mathlist}
    \item In the null recurrent case, for all $x\in \mathcal{X}$,  $\varrho^{-n}\mathbb{E}\big[Z_x(n)\big]\to 0$ as $n\to\infty$.
    \item In the positive recurrent case, if $\mathbb{P}(W>0)>0$ then
    \begin{equation*}
      \mathbb{E}\Big[\norm[\big]{\varrho^{-n} Z(n) - W\tilde{h}}_h\Big]\ =\ \sum_{x\in \mathcal{X}} \mathbb{E}\Big[\abs[\big]{\varrho^{-n}Z_x(n) - W\tilde{h}_x} \Big] h_x\ \underset{n\to\infty}{\longrightarrow} \ 0,
    \end{equation*}
    in other words, $\varrho^{-n}Z(n)\to  W\tilde{h}$ in $L^1(\mathbb{B}_h)$, and so in probability for the topology of the norm $\norm{\cdot}_h$.

If $\mathbb{P}(W>0)=0$ then $\lim_{n}\norm{\varrho^{-n}Z(n)}_h = \lim_n W_n = 0$ almost surely,  however $\varrho^{-n}Z(n)\to 0$ does not hold in $L^1(\mathbb{B}_h)$.
    \item In the positive recurrent case and if  $\mathbb{P}(W>0)>0$, then the convergence $\varrho^{-n}Z(n)\to  W\tilde{h}$ also holds almost surely in $\mathbb{B}_h$ -- \textit{i.e.} the limit
    \begin{equation}\label{eq:type-CV-as}
     \lim_{n\to\infty} \norm[\big]{\varrho^{-n} Z(n) - W\tilde{h}}_h \ =\ 0
    \end{equation}
    holds almost surely -- if one of the following conditions is satisfied:
    \begin{enumerate}[(a)]
      \item  $\displaystyle\sum_{x,y\in \mathcal{X}} \tilde{h}_x\mathrm{Var}(L^{(x)}_y)\tilde{h}_y^{-2} < \infty$; \label{ass:variance_condition}
      \item $\displaystyle \sum_{x,y\in \mathcal{X}} \tilde{h}_x \mathbb{E}\Big[L^{(x)}_y\log_+\big( L^{(x)}_y h_y^{-1}\tilde{h}_y^{-2} \big) \Big] h_y\ <\ \infty$.
      \label{ass:weaker_entropy_condition}
    \end{enumerate}
    Furthermore, (iii)-\ref{ass:weaker_entropy_condition} holds as soon as Condition~\eqref{eq:simple_condition} holds and the distribution $\pi= (\tilde{h}_x h_x)_x$ has finite entropy, \textit{i.e.} $-\sum_{x\in \mathcal{X}} \tilde{h}_xh_x \log\big( \tilde{h}_xh_x\big)
    <\infty$.
  \end{mathlist}
\end{theorem}

\begin{remark}\label{rmq:moy}
  Condition~(iii)-\ref{ass:variance_condition} is reminiscent of the
  condition that can be read
  in~\citet[Theorem~1]{moy_extensions_1967}. Translated to our
  notation, Moy's condition reads
  $\sum_x \tilde{h}_x \mathbb{E}[\langle L^{(x)},h\rangle^2]$, and she
  shows almost sure and $L^2$ convergence of the type distribution in
  the positive recurrent case. Note that
  Condition~(iii)-\ref{ass:weaker_entropy_condition} is easy to check
  in the case where~\eqref{eq:simple_condition} holds. However, we did
  not manage to show almost sure convergence of $\rho^{-n}Z(n)$ in all
  generality and the finite entropy condition appears, somewhat
  surprisingly, from our method of proof. Let us stress that it
    may well be that~\eqref{eq:type-CV-as} holds under the sole
    assumption $\mathbb{P}(W>0)$ in the positive recurrent case, but
    we did not manage to prove this.
\end{remark}

\begin{remark}\label{rmq:nummelin-tweedie}
  The notion of irreducibility, $R$-transience and $R$-(positive) recurrence of countable matrices also extend to kernels on countably generated measurable spaces -- we refer to \citet{Num84} for a general account.
  Perron--Frobenius theory also extends to this very general context, see \cite{Twe74}, and we would expect statement analogous to our
  main results to hold more generally, without relying on topological
  arguments -- for instance \cite{athreya_change_2000} obtains in
  his topology-free setting an analogue of Theorem~\ref{thm:ks_condition}.
  However, the techniques we use here are not easily adapted to general type spaces under minimal assumptions, since our proofs crucially rely on approximating the type space $\mathcal{X}$ by a finite set (for instance in Lemma~\ref{lem:restriction_finite} below, which is key to proving Theorem~\ref{thm:type_convergence}-(ii)).
\end{remark}

\section{Examples and counter-examples}\label{sec:examples}

A convenient way to build some examples of countable-type Galton--Watson processes with specific properties consists in first fixing a spinal Markov chain, then building the rest of the branching process around this.
Specifically, let us consider an irreducible, aperiodic, recurrent Markov chain $(X_n)$ with values on some countable space $\mathcal{X}$, with transition matrix denoted by $(p_{xy})$ and a stationary distribution $(\pi_x)$, assumed to be a probability distribution in the positive recurrent case.
Fix any $\varrho>1$ and any entry-wise positive $h=(h_x) \in \mathbb{R}^{\mathcal{X}}$, and define the matrix $M$ and the vector $\tilde{h}\in \mathbb{R}^{\mathcal{X}}$ via
\begin{equation}\label{eq:construction_ex}
  M_{xy}\ =\ \varrho\frac{p_{xy}h_x}{h_y},\quad \text{and}\quad \tilde{h}_x\ =\ \frac{\pi_x}{h_x}.
\end{equation}
Then by construction $M$ is a matrix with nonnegative entries and satisfies $Mh=\varrho h$ and $\tilde{h}^\transpose M=\varrho \tilde{h}^\transpose$.
Also, it is not hard to see that for all $x,y\in \mathcal{X}$ and $n\geq 1$, we have
\[
  \varrho^{-n}M^n_{xy}\ =\ \frac{h_x}{h_y}\mathbb{P}_x(X_n=y)\ \underset{n\to \infty}{\longrightarrow}\  \begin{cases}
    h_x\tilde{h}_y & \text{if }X\text{ is positive recurrent,}\\
    0 & \text{otherwise,}
  \end{cases}
\]
in other words, $\varrho, h$ and $\tilde{h}$ are exactly as in Assumption~\ref{ass:recurrence} for the matrix $M$, and $X$ is the associated spinal chain.
Now it remains to choose, for all $x\in \mathcal{X}$, any random vector $L^{(x)}=(L^{(x)}_y)_{y\in \mathcal{X}}$ to fix our type-$x$ offspring distribution, and as long as $\mathbb{E}[L^{(x)}_y]=M_{xy}$, we have built a countable-type Galton--Watson process with a given spinal chain.
Note that there is large freedom in the last step which consists in
fixing the offspring distribution, as everything that comes before
only depends on the mean reproduction matrix $M$. For instance,
  we can choose the $(L^{(x)}_y)_y$ to be Poisson random variables
  (independent or not) and tune the correlations as we see fit.
Using this construction, we now illustrate some noteworthy facts about our results.

\subsection*{Condition~\eqref{eq:simple_condition} may fail while~\eqref{eq:nsc} holds, even in the positive recurrent case}
Looking at Remark~\ref{rem:simple_implies_nsc}, it seems clear that~\eqref{eq:nsc} is weaker than~\eqref{eq:simple_condition}. Indeed, the two conditions are not equivalent, and we provide hereafter a simple example Condition~\eqref{eq:simple_condition} fail yet Condition~\eqref{eq:nsc} holds and so the additive martingale is uniformly integrable.

Consider a positive probability vector $(q_n)_{n\in \mathbb{N}}$ and define $X$ a Markov chain on $\mathcal{X}\coloneqq\{0\}\cup \{(n,m)\in \mathbb{N}^2 : 1\leq m \leq n\}$ with the following transition probabilities:
\begin{align*}
  p_{0,(n,1)}\ =\ q_n, \quad \text{and}\quad  p_{(n,n),0} &\ =\ 1 \qquad \forall n\geq 1,\\
  \text{and}\quad p_{(n,m),(n,m+1)} &\ =\ 1 \qquad  \forall n > m.
\end{align*}

We assume $\sum_n nq_n<\infty$, so that (considering the return time at $0$) $X$ is positive recurrent. Besides, we choose $(q_n)$ such that $\sum_{n} n^2q_n\ =\ \infty.$
It is easily seen that the  distribution $\pi$ such that
\[
 \pi_0\ =\ 1\quad\text{and}\quad  \pi_{(n,m)}\ =\ q_n \quad  \forall\ 1\leq  m\leq n,
\]
 defines a stationary distribution for the Markov chain.
We define $h_0 = 1$, and for any $1\leq m\leq n$, we now define $h_{(n,m)}$   to be equal to $\varrho^m$, so that by~\eqref{eq:construction_ex}, the left eigenvector is given by  $\tilde{h}_{(n,m)}=\varrho^{-m}q_n$.
Choosing any $\varrho>1$, we see that~\eqref{eq:construction_ex} implies that for $1\leq m < n$,
\[M_{(n,m),(n,m+1)}\ =\ 1,\qquad M_{0,(n,1)}\ =\  q_n,\qquad M_{(n,n),0}\ =\ \varrho^{n+1}.
\]
The above construction allows us to take $L^{(n,m)}_{(n,m+1)}$ to be Poisson distributed with mean $1$ (let us write $V$ for such a random variable).
This way, the sum in~\eqref{eq:simple_condition} is greater than
\[
 \sum_{1\leq m < n}  \varrho^{-m} q_n \varrho^{m+1} \mathbb{E}\Big[V\log_+ \big(\varrho^{m+1}V\big)\Big]\ =\  \varrho \sum_{1\leq m < n} q_n \mathbb{E}\Big[\log_+\big(\varrho^{m+1}(1+ V)\big)\Big].
\]
In turn, the latter is lower-bounded by 
\[
  \varrho \sum_{1\leq m < n} q_n (m+1) \log \varrho \ \geq \ \varrho\log \varrho \Big(\frac{1}{2} \sum_{n\geq 1} q_n n^2 - 1\Big)
\]
which is infinite by assumption. We still need to choose some other
distributions for $L^{(0)}$ and $L^{(n,n)}$ to define completely the
branching process; we choose them in a very simple way (for instance
vectors of independent Poisson random variables with parameter given
by the mean matrix). Let us now show that~\eqref{eq:nsc} holds
  in this example. We take $0$ as the starting point. For $n\geq 1$,
  with probability $q_n$ we have $X_1=(n,1)$ and conditional on this
  event, the variable $Y_1$ defined in~\eqref{eq:Yn_biased} has
  expectation:
  \[
    \mathbb{E}\big[ Y_1 \mid X_1 = (n,1) \big]\ =\
    \frac{1}{\mathbb{E}\big[L^{(0)}_{(n,1)}\big]}
    \mathbb{E}\Big[L^{(0)}_{(n,1)}\sum_{n'\geq 1}L^{(0)}_{(n',1)}
    h_{(n,1)} \Big]\ =\ 2 \rho\ \leq\ 2\rho^2,
  \]
  where we used the fact that $(L^{(0)}_{(n,1)})_{n\geq 1}$ are
  independent Poisson random variables with respective means
  $(q_n)_{n\geq 1}$. A similar calculation shows that for $2\leq k < n$,
\[
  \mathbb{E}\big[Y_k \mid X_1 = (n,1) \big]\ =\ \mathbb{E}\big[ \rho^{k+1}
  V^2 \big]\ =\ 2\rho^{k+1},
\]
where $V$ denotes a mean-one Poisson random variable. Finally, we can
similarly compute
\[
  \mathbb{E}\big[Y_n \mid X_1 = (n,1) \big]\ =\ 1+\rho^{n+1}\ \leq\ 2\rho^{n+1},
\]
so putting everything together in~\eqref{eq:nsc} and using the
assumption that $\sum_n nq_n<\infty$, we get
\begin{align*}
  \sum_{n\geq 1}q_n \mathbb{E}\Big[\log_+\Big( \sum_{k=1}^{n}\rho^{-k}
  Y_k \Big) \; \Big|\; X_1=(n,1) \Big]\
  & \leq\ \sum_{n\geq 1}q_n \sum_{k=1}^{n}\rho^{-k} \mathbb{E}\big[
    Y_k \; \big|\; X_1=(n,1) \big] \\
  &\leq\ \sum_{n\geq 1}q_n
    \sum_{k=1}^{n} \rho^{-k} \cdot (2 \rho^{k+1}) \ <\ \infty.
\end{align*}
Therefore, we see that there exist examples in
which~\eqref{eq:nsc} holds while~\eqref{eq:simple_condition} doesn't.

Note that there does not seem to be a reason for~\eqref{eq:nsc} to be necessary in the null-recurrent case to obtain the uniform integrability of $(W_n)$, but we did not find a counter-example for this situation.

\subsection*{Extinction is not characterized by $W=0$, even in the positive recurrent case}

As pointed out in Remark~\ref{rem:extinction}, we do not have $\big(\mathbb{P}(W>0)>0\big)\Rightarrow \big(\{W>0\}=\mathrm{Ext}^c\big)$ up to a null-measure set, as for Galton--Watson processes with a finite set of types.
This paragraph gives a counter-example in the positive recurrent supercritical framework.
This example already appears in~\citet[Section~4.1]{braunsteins_pathwise_2019}, but for the sake of completeness we still use our spinal chain construction to describe it.

Take $\varrho=4/3$ and let $X$ be the Markov chain with values in $\mathbb{N}$ whose transitions probabilities are given by $p_{n,n+1}=p_{n,0}=1/2$, for every $n\in \mathbb{N}$. $X$ is a positive recurrent Markov chain, $(\pi_n)=(2^{-(n+1)})$ being its stationary probability distribution.
Define the vector $h$ as $h_n=(2/3)^n$, such that $\tilde{h}_n=(3/4)^{n}/2$ and $M_{n,n+1}=1$, $M_{n,0}=2\cdot(2/3)^{n+1}$. We can now take the following offspring distributions:
\begin{align*}
 &L^{(n)}\ =\ e_{n+1} + \mathrm{Poi}\big(2\cdot (2/3)^{n+1}\big)e_0\qquad \forall n\geq 1, \\
 &L^{(0)}\ =\ \mathrm{Poi}(1)e_1 + \mathrm{Poi}(4/3)e_0\qquad n=0,
\end{align*} 
where the Poisson random variables are independent.

It is easy to check that~\eqref{eq:simple_condition} holds, so that $\mathbb{P}(W>0)>0$.
Theorem~\ref{thm:ks_condition} then implies that up to a null-measure set, we have
\[
  \mathrm{Ext}\ \subset\ \mathrm{LocExt}\ =\ \{W=0\},
\]
and we will see that in this example the first inclusion is not an equality.
Indeed, the event of seeing a type-$1$ offspring in the first generation whose descendants (note that the population descending from it survives forever, with exactly one descendant whose type goes to $\infty$) never produce any type-$0$ individuals is included in $\{W=0\}\cap\mathrm{Ext}^c$, so we have
\begin{align*}
  \mathbb{P}\big(\{W=0\}\cap \mathrm{Ext}^c\big)\ &\geq\ \mathbb{P}\big(L^{(0)}_1=1\big) \prod_{n\geq 0} \mathbb{P}\big( L^{(n)}_{0}=0\big)\\
  &=\ \exp\Big(-1 - 2\sum_{n\geq 0} \big(2/3\big)^{n+1}\Big)\\
  &=\ \mathrm{e}^{-5}\ >\ 0.
\end{align*}

\section{Proof of Theorem~\ref{thm:ks_condition}}
\label{sec:proof_ks_results}

Theorem~\ref{thm:ks_condition} puts together several classical ideas, the first of which is the so-called spinal change of measure.
In order to state this change of measure, up to enlarging our probability space, let us assume that our branching process keeps track of the whole genealogy of individuals.
Specifically, it will be convenient to use a Neveu--Ulam--Harris notation and identify individuals with elements of
\[
\mathcal{N}\ \coloneqq\ \bigcup_{n\geq 1}(\mathcal{X}\times\mathbb{N})^{n},
\] 
the set of finite sequences of elements of $\mathcal{X}\times\mathbb{N}$.
Individuals of type $x$ in generation 0 are labeled $(x,1),(x,2),\dots,(x,Z_x(0))$; the type-$y$ children of the particle labeled $(x,i)$ are then labeled $((x,i),(y,1)),((x,i),(y,2)),\ldots$, and so on.

With these labeling rules, we may view our branching process as a random forest of genealogical typed trees (each tree rooted at an individual in generation 0).
Let us write $\mathcal{T}_{ n}$ for the genealogy of our branching process consisting of all individuals up to generation $n$.
This genealogy $\mathcal{T}_{ n}$ can formally be seen as a random subset of $\mathcal{N}$, \textit{i.e.} a random variable with values in $\mathscr{T}\coloneqq\{0,1\}^{\mathcal{N}}$ endowed with the product $\sigma$-algebra -- a value of 1 at $v\in \mathcal{N}$ indicating the presence of individual $v$ in the genealogy, which we will write $v\in \mathcal{T}_{ n}$.

To build $\mathcal{T}_{ n}$ rigorously, first let $(\hat{L}^{v})_{v\in \mathcal{N}}$ be independent random variables, where the progeny $\hat{L}^{v}\overset{d}{=}L^{(x)}$ if $v$ is of type $x$ -- in what follows, we will use $x^v$ to denote the type of $v$.
Then, we define $\mathcal{T}_{ n}$ as the set of individuals $v\in \mathcal{N}$ satisfying that $v=((x_0,i_0),\ldots,(x_m,i_m))$ for some $m\leq n$, and writing $v_k=((x_0,i_0),\ldots,(x_k,i_k))$ for the ancestor of $v$ at generation $k$:
\[
  i_0\ \leq\ Z_{x_0}(0) \quad \text{and} \quad i_{k} \ \leq\ \hat{L}^{v_{k-1}}_{x_k}, \qquad 1\leq k \leq m,
\]
where we recall that for all $x\in\mathcal{X},\ \hat{L}^{v_{k-1}}_{x}$ is the number of children of $v_{k-1}$ bearing type $x$. 
With this construction, we can define $Z(n)\in \mathbb{N}^{\mathcal{X}}$ as the count of individuals of each type at generation $n$. That is, $Z_x(n)$ is the cardinal of the set of particles with type $x$ in $
\mathcal{T}_{ n}\setminus \mathcal{T}_{ n-1}$, for all $x\in \mathcal{X}$.
Thus defined, $(Z(n))_n$ is indeed the countable-type Galton--Watson process associated with the offspring distributions $(L^{(x)})_{x\in \mathcal{X}}$.

The spinal change of measure relies on the next lemma, which deals with random genealogies \emph{with a marked individual} -- note that a marked forest can be viewed as a random variable $(\mathcal{T},v)$ with values in $\mathscr{T}^*\coloneqq\mathscr{T}\times \mathcal{N}$. 
Before stating the result, let us recursively build the ``Kesten tree'' (which is technically a forest here) associated with our branching process. We emphasize that this construction relies heavily on that done in~\cite{lyons_pemantle_peres_1995} in the scalar setting.
\begin{itemize}
  \item Define $(\mathcal{T}^{*}_0, v_0)$ as the population at generation $0$ (so $\mathcal{T}^{*}_0$ is deterministic and equal to the $\mathcal{T}_0$ defined above) with a marked vertex $v_0$. This marked particle is chosen at random 
  such that for all $u\in \mathcal{T}_0$, $\mathbb{P}(v_0=u)$ is proportional to $h_{x^u}$.
 
 \item Choose a random type $X_{1}\in \mathcal{X}$ by taking a step of the spinal Markov chain $X$ with transition kernel given in~\eqref{eq:spine} and starting from $X_0=x^{v_0}$.
  \item For all $n\geq 0$, conditional on $(\mathcal{T}^{*}_{ n}, v_n)$, we build $(\mathcal{T}^{*}_{ n+1}, v_{n+1})$ as follows (each step is independent of everything else):
  \begin{itemize}

    \item Let $(\widetilde{L}^{n+1}, X_{n+1})$ be a random vector and type with the biased distribution 
      \[
  \mathbb{P}\Big(\big(\widetilde{L}^{n+1}, X_{n+1}\big)\in \cdot\ \vert\   X_{n}=x \Big)\ =\ \frac{1}{\varrho} \sum_{y\in \mathcal{X}}\frac{h_y}{ h_x}\mathbb{E}\Big[L^{(x)}_{y}\, \mathds{1}\big((L^{(x)},y)\in \cdot \big) \Big].
  \]
  
    \item $\mathcal{T}^{*}_{ n+1}$ is build from $\mathcal{T}^{*}_{
        n}$ by letting $v_n$ reproduce according to
      $\widetilde{L}^{n+1}$ (\textit{i.e.} $v_n$ has $\widetilde{L}^{n+1}_x$ children of type $x$), and all other individuals $v$ at generation $n$ reproduce independently according to a copy of $L^{(x^v)}$.
    
    \item Finally, $v_{n+1}$ is chosen uniformly among the $\widetilde{L}^{n+1}_{X_{n+1}}\geq 1$ type-$X_{n+1}$ children of $v_n$.
  \end{itemize}
\end{itemize}

\begin{lem}\label{lem:spinal-change-of-measure}
  Let $(\widetilde{\mathcal{T}}_{ n}, \tilde{v}_n)$ be the random marked forest whose distribution is given by
  \[
    \mathbb{E}\big[F(\widetilde{\mathcal{T}}_{ n},\tilde{v}_n)\big]\ =\ \mathbb{E}\bigg[\frac{1}{\varrho^n \langle Z(0),h\rangle}\sum_{v\in \mathcal{T}_{n}\setminus \mathcal{T}_{n-1}} h_{x^v}F(\mathcal{T}_{ n},v)\bigg],
  \]
  for any measurable and nonnegative function $F:\mathscr{T}^*\to \mathbb{R}$, and we recall $\mathcal{T}_{n}\setminus \mathcal{T}_{n-1}$ denotes the vertices of $\mathcal{T}_{n}$ living at generation exactly~$n$.
  Then $(\widetilde{\mathcal{T}}_{ n}, \tilde{v}_n)$ has the distribution of $(\mathcal{T}^*_{ n}, v_n)$.
\end{lem}

\begin{proof}
  We prove this by induction.
  The equality in distribution is clear for $n=0$ by construction.
  We now assume it holds for some $n\geq 0$, and fix a measurable, nonnegative function $F$.
  For any individual $u\in \mathcal{T}_{n+1}$, we will denote by $\mathrm{Ch}(u)$ (resp.\ $\mathrm{Ch}_y(u)$) the set of children (resp.\ type-$y$ children) of $u$, and recall that $L^u_y$ is the cardinal of $\mathrm{Ch}_y(u)$.
  We compute:
  \begin{align*}
     \mathbb{E}\bigg[&\frac{1}{\varrho^{n+1} \langle Z(0), h\rangle}\sum_{v\in \mathcal{T}_{n+1}\setminus\mathcal{T}_n} h_{x^v}F(\mathcal{T}_{ n+1},v)\bigg]\\
     &\quad =\ \mathbb{E}\bigg[\frac{1}{\varrho^{n} \langle Z(0), h\rangle}\sum_{u\in \mathcal{T}_{n}\setminus \mathcal{T}_{n-1}} \frac{1}{\varrho} \sum_{v\in \mathrm{Ch}(u)} h_{x^v}F(\mathcal{T}_{ n+1},v)\bigg]\\
    &\quad =\ \mathbb{E}\bigg[\frac{1}{\varrho^{n}\langle Z(0), h\rangle}\sum_{u\in \mathcal{T}_{n}\setminus \mathcal{T}_{n-1}} h_{x^u} \sum_{y\in \mathcal{X}} \underbrace{ \frac{M_{x^uy}h_y}{\varrho h_{x^u}} }_{=p_{x^uy}} \cdot \frac{L^u_y}{M_{x^uy}} \cdot \frac{1}{L^u_y} \sum_{v\in \mathrm{Ch}_y(u)} F(\mathcal{T}_{ n+1},v)\bigg].
  \end{align*}
  It suffices to note that by the induction hypothesis and the recursive construction of $(\mathcal{T}^*_{ n+1},v_{n+1})$ from $(\mathcal{T}^*_{ n},v_{n})$, the last expression above is exactly $\mathbb{E}\big[F(\mathcal{T}^*_{ n+1}, v_{n+1})\big]$, which concludes the proof.
\end{proof}

This result can now be translated in terms of the martingale $(W_n)_{n\geq 0}$.
Here and in the following, we let $(\mathcal{F}_n)$ denote the natural filtration associated to $(Z(n))$.
By Kolmogorov's extension theorem, there exists a probability measure $\mathbb{Q}$ on $\mathcal{F}_\infty = \sigma(\bigcup_n \mathcal{F}_n)$ satisfying that (writing respectively $\mathbb{P}_n$ and $\mathbb{Q}_n$ for the restriction of $\mathbb{P}$ and $\mathbb{Q}$ to $\mathcal{F}_n$) $\mathbb{Q}_n \ll \mathbb{P}_n$, with Radon--Nikodym derivative proportional to $W_n=\varrho^{-n}\langle Z(n), h\rangle$.
The fact that $\langle Z(n), h\rangle=\sum_{v\in \mathcal{T}_n} h_{x^v}$ allows us to describe easily the distribution of $(W_n)$ under $\mathbb{Q}$ (in the upcoming Corollary~\ref{cor:law-of-W-under-Q}).
Note that we will never need to use expectation under $\mathbb{Q}$, so that $\mathbb{E}$ always denotes expectation under $\mathbb{P}$.

\begin{figure}
 \centering
\resizebox{0.83\textwidth}{!}{

\begin{tikzpicture}[x=0.75pt,y=0.75pt,yscale=-1,xscale=1]

\draw [line width=1.5]    (10.99,6.44) -- (9.01,349.44) ;
\draw [shift={(11.01,2.44)}, rotate = 90.33] [fill={rgb, 255:red, 0; green, 0; blue, 0 }  ][line width=0.08]  [draw opacity=0] (11.61,-5.58) -- (0,0) -- (11.61,5.58) -- cycle    ;
\draw [line width=1.5]    (0.74,349.35) -- (20,349.73) ;
\draw [line width=1.5]    (0.74,240.35) -- (20,240.73) ;
\draw [line width=1.5]    (0.74,129.35) -- (20,129.73) ;
\draw [line width=1.5]    (0.74,39.35) -- (20,39.73) ;
\draw (399,191.06) node  {\includegraphics[width=387pt,height=267pt]{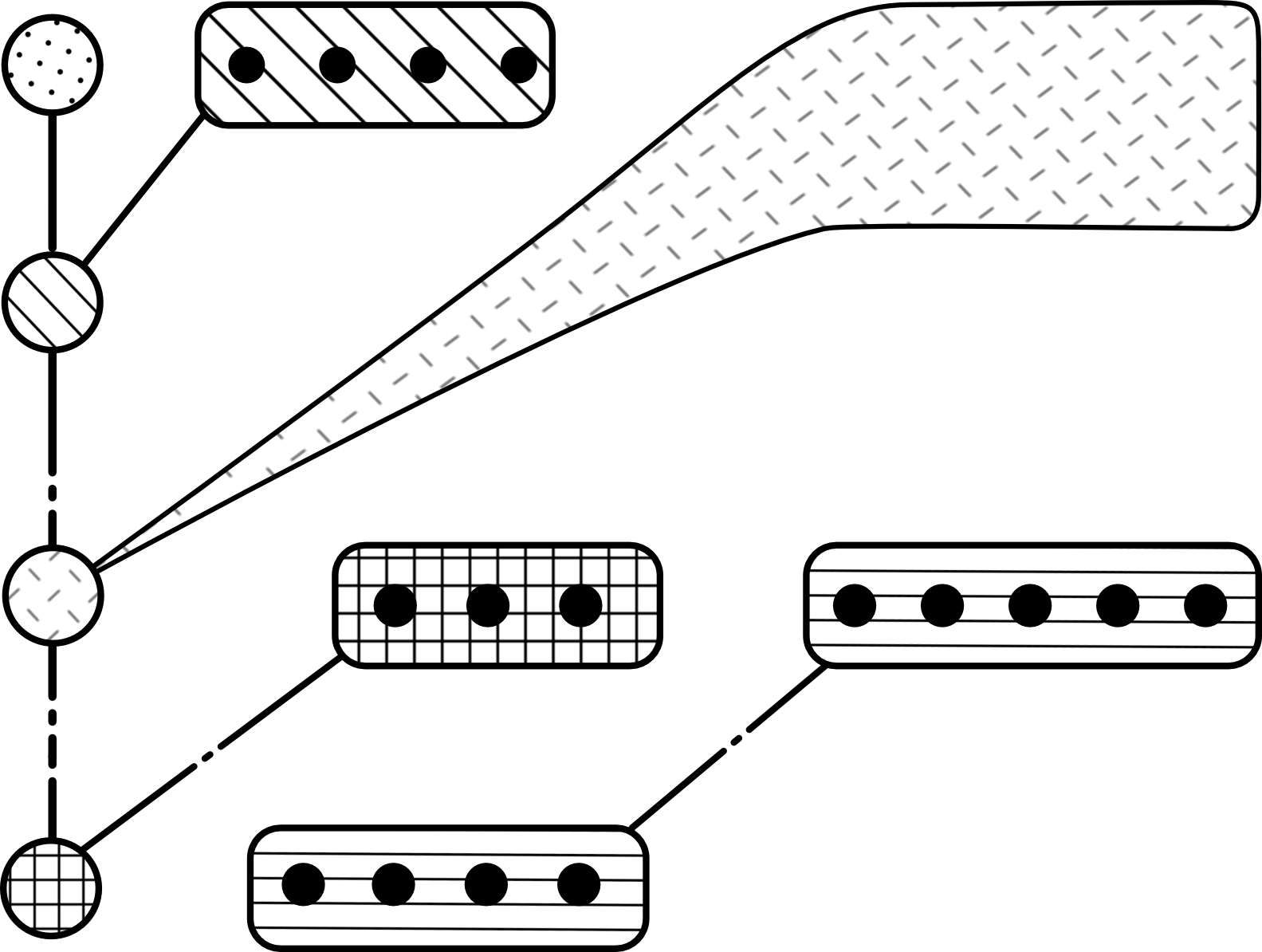}};

\draw (80,226.4) node [anchor=north west][inner sep=0.75pt]    {\Large$v_{k} ,\ X_{k}$};
\draw (193,119.4) node [anchor=north west][inner sep=0.75pt]    {\Large$v_{n-1} ,\ X_{n-1}$};
\draw (80,338.4) node [anchor=north west][inner sep=0.75pt]    {\Large$v_{0} ,\ X_{0}$};
\draw (280,268.4) node [anchor=north west][inner sep=0.75pt]    {\Large$Z^{( 1)}(k-1)$};
\draw (479,269.4) node [anchor=north west][inner sep=0.75pt]    {\Large$Z^{0}( k)$};
\draw (465,105.4) node [anchor=north west][inner sep=0.75pt] {\Large$Z^{(k+1)}(m), 0\leq m \leq n-(k+1)$};
\draw (228,65.4) node [anchor=north west][inner sep=0.75pt]  [color={rgb, 255:red, 0; green, 0; blue, 0 }  ,opacity=1 ]  {\Large$Z^{(n)}(0)$};
\draw (80,28) node [anchor=north west][inner sep=0.75pt]    {\Large$v_{n} ,\ X_{n}$};
\draw (419,335.4) node [anchor=north west][inner sep=0.75pt]    {\Large$Z^{( 0)}( 0)$};
\draw (28,35) node [anchor=north west][inner sep=0.75pt]    {\Large$n$};
\draw (28,120) node [anchor=north west][inner sep=0.75pt]    {\Large$n-1$};
\draw (26,232) node [anchor=north west][inner sep=0.75pt]    {\Large$k$};
\draw (26,337.4) node [anchor=north west][inner sep=0.75pt]    {\Large$0$};
\end{tikzpicture}
}
\caption{Illustration of the process $(Z^{(n)}(m))_{n,m}$. The
    spinal particles and their types are pictured in the left-hand
    side lineage. The particle labeled $v_{n-1}$ has type $X_{n-1}$ and
    produces a biased (positive) number of offspring. The particle
    labeled $v_n$ is chosen among the type-$X_n$ progeny of $v_{n-1}$,
    whereas the rest of the progeny makes up the first generation
    $Z^{(n)}(0)$ of an independent copy of the initial process. The
    total population at generation $n$ is constituted from the
    descendants of the successive immigration waves
    $(Z^{(k)}(n-k))_{k\leq n}$.} \label{fig:spine}
\end{figure}
Let $(X_n)_{n\geq 0}$ be the spinal chain, started from an initial
condition $X_0$ drawn at random such that $\mathbb{P}(X_0=x) = Z_x(0)h_x/\langle Z(0), h\rangle$, and let us expand on the construction of $(Y_n)_{n\geq 1}$ in~\eqref{eq:Yn_biased} by defining a sequence of $\mathbb{N}^{\mathcal{X}}$-valued random processes $(Z^{(n)}(m))_{n,m\geq 0}$, such that:
\begin{itemize}
  \item The $(Z^{(n)}(0))_{n\geq 0}$ are independent conditional on $\mathcal{G}=\sigma(X_n,n\geq 0)$.
  \item For $n=0$, set $Z^{(n)}(0)=Z(0) - e_{X_0}$, where we recall the notation $e_y=(\mathds{1}_{x=y})_{x}\in \mathbb{N}^{\mathcal{X}}$.
  \item For $n\geq 1$, the distribution of $(Z^{(n)}(0), X_n)$ conditional on $\mathcal{G}$ and $X_{n-1}=x$ is biased such that   for every nonnegative measurable function $F:\mathbb{N}^{\mathcal{X}}\times \mathcal{X}\to \mathbb{R}_+$,
  \[
\mathbb{E}\Big[F\big(Z^{(n)}(0), X_n\big)\ \vert\   X_{n-1}=x \Big]\ =\ \frac{1}{\varrho} \sum_{y\in \mathcal{X}}\frac{h_y}{ h_x}\mathbb{E}\Big[L^{(x)}_{y}F\big(L^{(x)}-e_{y}, y\big) \Big].
  \]
  \item Finally, conditional on $(Z^{(n)}(0))_{n\geq 0}$ and independently of all other randomness, let $(Z^{(n)}(m))_{m\geq 0}$ be, for each $n$, a branching process started from $Z^{(n)}(0)$, with the same Markov kernel as the original Galton--Watson process $(Z(m))_{m\geq 0}$. We refer to Figure~\ref{fig:spine} for clarity.
\end{itemize}
With this definition, note that a version of $(Y_n)_n$ defined in~\eqref{eq:Yn_biased} is given by $Y_n=h_{X_n}+\langle Z^{(n)}(0),h\rangle$.

Now let $(\widetilde{W}_n)_{n\geq 0}$ denote the process defined by
\begin{equation}\label{eqdef:tildeW}
  \widetilde{W}_n\ =\ \varrho^{-n}\bigg(h_{X_n} + \sum_{k=0}^{n}\langle Z^{(k)}(n-k), h \rangle\bigg)\ =\ \varrho^{-n}h_{X_n} + \sum_{k=0}^{n}\varrho^{-k}\Big(\frac{\langle Z^{(k)}(n-k), h \rangle}{\varrho^{n-k}}\Big).
\end{equation}
The particles in the population process at generation $n$ descend from one of the spinal individuals $v_0,\ldots, v_{n}$.
Note that for each $k$, the process $W^{(k)}_n \coloneqq \varrho^{-n}\langle Z^{(k)}(n), h \rangle$ is the additive martingale of the process $Z^{(k)}$, and appears in the definition of $\widetilde{W}$. The following consequence of Lemma~\ref{lem:spinal-change-of-measure} provides a justification for this definition.

\begin{cor}\label{cor:law-of-W-under-Q}
  The distribution of $(W_n)$ under $\mathbb{Q}$ is that of $(\widetilde{W}_n)$.
\end{cor}
\begin{proof}
For any nonnegative measurable $F:\mathbb{R}\mapsto \mathbb{R}_+$, the function $\Psi:(\mathcal{T}_n, v_n) \longmapsto F\big(\varrho^{-n}\sum_{v\in \mathcal{T}_n\setminus \mathcal{T}_{n-1}}h_{x^v}\big)$ is nonnegative and measurable on $\mathscr{T}^*$. By Lemma~\ref{lem:spinal-change-of-measure},
\begin{align*}
\mathbb{E}\big[F(\widetilde{W}_n)\big]\ &=\ \mathbb{E}\big[\Psi(\mathcal{T}^*_n, v_n)\big]\\
& =\ \mathbb{E}\bigg[\frac{1}{\varrho^n \langle Z(0),h\rangle}\sum_{v\in \mathcal{T}_{n}\setminus \mathcal{T}_{n-1}} h_{x^v}\Psi(\mathcal{T}_{ n},v)\bigg], \\
&=\  \mathbb{E}\Big[\frac{W_n}{\langle Z(0),h\rangle}  F(W_n)\Big]\\
&=\ \mathbb{Q}\big[F(W_n)\big], 
\end{align*}
where we used the construction of $\mathbb{Q}$ in the last line.
\end{proof}

As in~\citet{lyons_pemantle_peres_1995}, The two following classic results are key to proving a Kesten--Stigum theorem through the spinal change of measures.
We state them without proof: the first one is deduced for instance from~\citet[Theorem~5.3.3]{Dur10}, with $\mu=\mathbb{Q}$, $\nu=\mathbb{P}$ and $A=\{W>0\}$; the second one is a simple consequence of the Borel--Cantelli lemmas.

\begin{lem}\label{lem:dichotomy-tildeW}\hfill
  \begin{mathlist}
    \item If $\limsup_n \widetilde{W}_n < \infty$ a.s, then $\mathbb{E}[W]=\mathbb{E}[W_0]$.
    \item If $\limsup_n \widetilde{W}_n = \infty$ a.s, then $\mathbb{E}[W]=0$.
  \end{mathlist}
\end{lem}

\begin{lem}\label{lem:dichotomy-expectation}
  Let $(X_n)_{n\geq 1}$ be i.i.d.\ nonnegative random variables
  on some probability space $(\Omega, \mathcal{F}, \mathbb{P})$. Then almost surely:
  \[
    \limsup_{n\to\infty} \frac{X_n}{n}\ =\ \begin{cases}
      0 \quad & \text{if }\ \mathbb{E}[X_1]\ <\ \infty,\\
      \infty & \text{if }\ \mathbb{E}[X_1]\ =\ \infty.
    \end{cases}
  \]
  If $\mu$ is a (non-necessarily finite) nonnegative measure on
    $(\Omega, \mathcal{F})$ and if the $(X_n)_{n\geq 1}$ are only $\mu$-identically
    distributed in the sense that the push-forward measures
    $\mu\circ X_n^{-1}$ do not depend on $n$, then
    $\int X_1 \, d\mu <\infty$ still implies that
    $\limsup_{n\to\infty} X_n/n=0$ holds $\mu$-almost everywhere.
\end{lem}

We will in fact apply Lemma~\ref{lem:dichotomy-expectation} to
  $\log_+(Y_n)$, where $(Y_n)_n$ are positive random variables. Note
  that the conclusions of the lemma are adapted using the following
  remark: for a sequence of positive real numbers $(y_n)_{n\geq 1}$,
  $\log_+(y_n)/n\to 0$ if and only if $y_n$ grows subexponentially,
  \textit{i.e.} if
  \begin{equation}\label{eq:subexponential_growth}
    \forall \gamma>1, \quad \gamma^{-n}y_n \ \longrightarrow\ 0.
  \end{equation}
  Conversely, $\limsup_{n\to\infty} \log_+(y_n)/n = \infty$ if and only if for
  any $\gamma>1$, we have $y_n \geq \gamma^{n}$ for infinitely many
  $n$ indices. 

We are now ready to prove the different parts of Theorem~\ref{thm:ks_condition}.

\begin{prop}\label{prop:thm-ks-part1}
  We have $\limsup_n \widetilde{W}_n < \infty$ a.s.\ if~\eqref{eq:simple_condition} or~\eqref{eq:nsc} is satisfied.
\end{prop}

\begin{proof}
  Let us first show that it is sufficient to study $(\widetilde{W}_n)$ for $Z(0)=e_{x_0}$, for a fixed $x_0\in \mathcal{X}$.
  For any other $Z(0)$, defining $T_{x_0}=\min\{n \geq 0 : X_n = x_0\}$, which is finite a.s.\ since the spinal chain is recurrent, by~\eqref{eqdef:tildeW} we can write
  \[
    \widetilde{W}_n\ =\ \sum_{k=0}^{T_{x_0}-1} \varrho^{-k}\Big(\frac{\langle Z^{(k)}(n-k), h\rangle}{\varrho^{n-k}}\Big) \;+\; \widetilde{W}^{0}_{n-T_{x_0}}, \qquad n \geq T_{x_0},
  \]
  where, by the strong Markov property $(\widetilde{W}^{0}_{m})_{m\geq 0}$ is distributed as $\widetilde{W}_{m}$ with initial condition $Z(0)=e_{x_0}$.
  Now the sum in the display above is an almost surely finite sum of processes that converge almost surely, since conditional on $Z^{(k)}(0)$, the process $(\varrho^{-m}\langle Z^{(k)}(m), h \rangle)_{m\geq 0}$ is a nonnegative martingale.
  This shows that $\limsup_n\widetilde{W}_n=\infty$ if and only if $\limsup_n\widetilde{W}^0_n=\infty$; thus for the rest of the proof, we may assume that $Z(0)=e_{x_0}$.
  
  The description of $(\widetilde{W}_n)$ given in~\eqref{eqdef:tildeW} is convenient for this problem since conditional on $\mathcal{G}'\coloneqq\sigma(X_n,Z^{(n)}(0),n\geq 0)$, we see that the process
  \[
    W'_n \ \coloneqq\ \widetilde{W}_n - \varrho^{-n}h_{X_n}\ =\ \varrho^{-n}\sum_{k=0}^{n}\langle Z^{(k)}(n-k),h \rangle
  \]
  is a nonnegative submartingale.
  To see this, note that for a fixed $k\geq 0$, the process $(\rho^{-n}\langle Z^{(k)}(n-k),h \rangle)_{n\geq k}$ is a martingale, so we compute
  \[
    \mathbb{E}\big[W'_{n+1} \mid \mathcal{G}'\vee \mathcal{F}_n\big]\ =\ W'_n + \varrho^{-(n+1)}\langle Z^{(n+1)}(0),h\rangle.
  \]
  Therefore, if we can show that
  \begin{equation}\label{eq:for-finite-W'}
    \sum_{n\geq 1}\varrho^{-n}\langle Z^{(n)}(0),h\rangle \ <\ \infty \qquad a.s,
  \end{equation}
  we deduce from standard martingale arguments that $W'=\lim_{n\to \infty}W'_n$ exists and is finite.
  
  First, we assume that~\eqref{eq:simple_condition} is satisfied.
  Since $\pi=(\tilde{h}_xh_x)_x$ is a stationary distribution for the
  spinal Markov chain, the random variables $(Y_n)_{n\geq 1}$
    are $\mu$-identically distributed in the sense of
    Lemma~\ref{lem:dichotomy-expectation}, with $\mu \coloneqq \sum_{x\in \mathcal{X}} \pi_x \mathbb{P}_x(\cdot)$. Note that in the recurrent case, the measure $\mu$ is non-necessarily finite.
  We also compute
  \begin{align*}
   \int_{\mathcal{X}} \log_+(Y_n) d\mu\ & =\ \sum_{x\in \mathcal{X}} \pi_x \mathbb{E}_x\big[\log_+(Y_1)\big]\\
    &=\ \sum_{x\in \mathcal{X}}\tilde{h}_x h_x \sum_{y\in \mathcal{X}} \frac{M_{xy}h_y}{\varrho h_x} \cdot \frac{1}{M_{xy}}\mathbb{E}\Big[L^{(x)}_y\log_+\langle L^{(x)},h\rangle\Big]\\
    &=\  \frac{1}{\varrho}\sum_{x\in \mathcal{X}}\tilde{h}_x \mathbb{E}\Big[\langle L^{(x)},h\rangle\log_+\langle L^{(x)},h\rangle\Big]\ <\ \infty,
  \end{align*}
Therefore, by Lemma~\ref{lem:dichotomy-expectation}, $(Y_n)_n$
  grows subexponentially, \textit{i.e.}~\eqref{eq:subexponential_growth} is satisfied $\mu$-almost everywhere, which implies that this is true $\mathbb{P}_{x_0}$-almost surely.
  In the construction of $(\widetilde{W}_n)$, we can write $Y_n = h_{X_n}+\langle Z^{(n)}(0),h\rangle$, so that both $\langle Z^{(n)}(0),h\rangle$ and $h_{X_n}$ grow subexponentially, implying~\eqref{eq:for-finite-W'} and also $\abs{\widetilde{W}_n - W'_n} = \varrho^{-n}h_{X_n}\to 0$.
  Putting everything together, we get that $\lim_{n\to \infty}\widetilde{W}_n$ exists and is finite.
  
  We now assume that~\eqref{eq:nsc} is satisfied for $x=x_0$.
  Let us denote by $T_0=0$ and $T_1,T_2,\dots$ the successive return times to $x_0$ for the spinal chain $(X_n)$, and for $n\geq 0$, let us define $k_n$ such that $T_{k_n-1}\leq n < T_{k_n}$.
  Grouping terms, the sum in~\eqref{eq:for-finite-W'} is bounded by
  \begin{equation}\label{eqdef:Y_prime}
      \sum_{k=1}^{k_n}\sum_{j=T_{k-1}+1}^{T_k}\varrho^{-j}Y_j\ =\  \sum_{k=1}^{k_n}\varrho^{-T_{k-1}} \Big(\sum_{j=1}^{T_{k}-T_{k-1}}\varrho^{-j}Y_{T_{k-1}+j}\Big)\ =\ \sum_{k=1}^{k_n}\varrho^{-T_{k-1}} Y'_k,
  \end{equation}
  where the $Y'_k$ are i.i.d., with $Y'_1=\sum_{j=1}^{T_1}\varrho^{-j}Y_j$.
  Therefore, by Lemma~\ref{lem:dichotomy-expectation}, Condition~\eqref{eq:nsc} implies that $(Y'_k)$ grows subexponentially, and since necessarily $T_k\geq k$, this shows that~\eqref{eq:for-finite-W'} holds.
  Therefore, $\lim_n W'_n$ exists and is finite, and it remains to show that $\varrho^{-n}h_{X_n}\to 0$.
  But by definition of $k_n$, we have $\varrho^{-n}h_{X_n}\leq \varrho^{-n}Y_n \leq \varrho^{-T_{k_n-1}}Y'_{k_n}$, so we conclude that $\lim_n \widetilde{W}_n$ exists and is finite.
\end{proof}

\begin{prop}\label{prop:ks-necessary-part}
  If $M$ is positive recurrent and~\eqref{eq:nsc} does not hold, we have 
  \[
  \limsup_{n\to\infty}\widetilde{W}_n\ =\ \infty,
  \] up to a null measure set.
\end{prop}

\begin{proof}
  Here we take the same notation as in the proof above, and assume that~\eqref{eq:nsc} does not hold.
  Since the $(Y'_k)$ are i.i.d., by Lemma~\ref{lem:dichotomy-expectation} we get that almost surely, for any $\gamma>1$, we have $Y'_k \geq \gamma^{k}$ for infinitely many $k$.
  Let us fix $\gamma > \varrho^{c}$, where $c=\mathbb{E}[T_1]$, which is finite because $(X_n)$ is assumed to be positive recurrent.
  Note that by the definition of $Y'_k$ in~\eqref{eqdef:Y_prime}, if $Y'_k \geq \gamma^{k}$ holds, then necessarily we have
  \begin{equation}\label{eq:max_Y_j}
    \max_{T_{k-1}+1\leq j\leq T_k} \big(\varrho^{-j}Y_{j}\big)\ \geq\ \frac{\gamma^{k}\varrho^{-T_{k-1}}}{(T_{k}-T_{k-1})}.
\end{equation}
  Because $\mathbb{E}[T_1]=c<\infty$, we have $T_k\sim ck$ a.s., and furthermore (again by Lemma~\ref{lem:dichotomy-expectation}), for $k$ large enough, we will have $T_{k}-T_{k-1}\leq k$, so with our choice of $\gamma$, the right-hand-side in~\eqref{eq:max_Y_j} tends to $+\infty$ as $k\to \infty$.
  Summing up, we have shown that almost surely, there exists a sequence $j_m\to\infty $ such that $\varrho^{-j_m}Y_{j_m}\to \infty$. At last, the fact that by definition $\widetilde{W}_n \geq \varrho^{-n}Y_n$ concludes the proof.
\end{proof}

Putting together Proposition~\ref{prop:thm-ks-part1} and Proposition~\ref{prop:ks-necessary-part}, we have shown that in the positive recurrent setting, \eqref{eq:nsc} is equivalent to $\limsup_n \widetilde{W}_n <\infty$, which by Lemma~\ref{lem:dichotomy-tildeW} is equivalent to $\mathbb{E}[W]=\langle Z(0), h \rangle$. We have in addition that~\eqref{eq:simple_condition} is in general only a sufficient condition -- see Section~\ref{sec:examples} for an example satisfying~\eqref{eq:nsc} but not~\eqref{eq:simple_condition}.

\medskip

To complete the proof of Theorem~\ref{thm:ks_condition}, it remains to tackle the equivalence of the three items~\ref{eq:point_1_ext}-\ref{eq:point_3_ext}.
First, we have already seen in Lemma~\ref{lem:dichotomy-tildeW} that $\mathbb{E}[W]$ is either $0$ or $\langle Z(0), h\rangle$, which shows that~\ref{eq:point_1_ext} and~\ref{eq:point_2_ext} are equivalent.
Let us see that~\ref{eq:point_3_ext} implies~\ref{eq:point_2_ext}. We assume~\ref{eq:point_3_ext}, fix $x\in \mathcal{X}$ and note that since we are in the supercritical regime, we have $M^d_{x,x}>1$ for some $d$, whose existence comes from the definition of $\varrho$ in~\eqref{eq:rho}.
Then started from $Z(0)=e_x$, the process $(Z_x(nd))_{n\geq 0}$ stochastically dominates a supercritical (single-type) branching process, which therefore has a nonzero probability of survival.
But on the event of survival, there is no local extinction, so $W>0$ by assumption and~\ref{eq:point_2_ext} is therefore satisfied.

We now only need to show that~\ref{eq:point_2_ext} implies~\ref{eq:point_3_ext}.
Let us assume~\ref{eq:point_2_ext} holds.
First, we use Theorem~\ref{thm:type_convergence}.(ii), which shows in particular that for a fixed $x\in \mathcal{X}$, we have $\varrho^{-n}Z_x(n)\to W\tilde{h}_x$ in probability.
Note that we are not using a circular argument here since the proof of Theorem~\ref{thm:type_convergence} in Section~\ref{sec:proof_type_convergence} only uses the fact that $\mathbb{P}(W>0) \iff \mathbb{E}[W]=\langle Z(0), h\rangle$.
We can define an increasing subsequence $(n_k)_{k\geq 1}$ such that $\varrho^{-n_k}Z_x(n_k)\to W\tilde{h}_x$ almost surely, and this show that on the event $\{W>0\}$, there cannot be local extinction.
In other words we have
\[
  \{W>0\}\ \subset\ \mathrm{LocExt}^c.
\]

It remains to derive the converse inclusion.
By definition of the local extinction event we have $\mathrm{LocExt}^c = \{\exists x_0,\ \limsup_{n}Z_{x_0}(n)>0\}$. We fix $x_0\in \mathcal{X}$ and focus on the event $\{\limsup_{n}Z_{x_0}(n)>0\}$.
Let us fix some integer $K>0$.
Under Assumption~\ref{ass:recurrence} one can find $d\in \mathbb{N}$ such that $M^d_{x_0,x_0}>K$, and so in particular, started from $Z(0)=e_{x_0}$, we have $\mathbb{P}(Z_{x_0}(d)\geq K)=c>0$.

Define the recursive sequence of stopping times $\tau_0=\inf\{n>0,\ Z_{x_0}(n)\geq 1\}$ and for $k\geq 0$,  $\tau_{k+1}=\inf\{n>\tau_k+d,\ Z_{x_0}(n) \geq 1\}$.
The $(\tau_k)_k$ are all almost surely finite on $\{\limsup_{n}Z_{x_0}(n)>0\}$ and by the strong Markov property, the $(Z_{x_0}(\tau_n+d))_n$ are independent conditional on $\mathcal{G}\coloneqq\sigma(Z_{x_0}(\tau_n), n\geq 0)$ and all satisfy
\[
 \mathbb{P}\big(Z_{x_0}(\tau_n+d) > K \mid \mathcal{G}\big)\ \geq\ c.
\]

The conditional Borel--Cantelli lemma then yields that on $\{\limsup_{n}Z_{x_0}(n)>0\}$, $\limsup_n Z_x(n)>K$ holds almost surely. As a result the stopping time $t_K=\inf\{n>0,\ Z_{x_0}(n)\geq K\}$ is almost surely finite. By the branching property at generation $t_K$, each of the type-$x_0$ individuals at generation $n$ (name them $u$) is the ancestor to an independent multi-type branching process with same offspring distribution to which we associate an additive martingale $W_n^{(u)}$ converging pointwise to a copy $W^{(u)}$ of $W$ (independent of the other $W^{(v)},\ v\neq u$).
The initial martingale being smaller than the sum of the $W_n^{(u)}$, we have $W=0$ only if for all $u$, $W^{(u)}=0$.
Since at generation $t_K$ there are at least $K$ type-$x_0$ individuals, we then have
\begin{align*}
  \mathbb{P}\Big(\limsup_{n\to\infty}Z_{x_0}(n)>0\ \text{and}\ W=0\Big)\ &\leq \ \mathbb{P}\big(t_K <\infty \text{ and }W=0\big)\\
  &\leq\ \mathbb{P}\big(t_K <\infty) \mathbb{P}\big(W=0 \mid Z(0)= e_{x_0}\big)^K.
\end{align*}
Under~\ref{eq:point_2_ext} and because we assumed irreducibility, $\mathbb{P}(W=0)<1$ and
the upper bound tends to $0$ as $K\to \infty$, which yields $\{\limsup_{n}Z_{x_0}(n)>0\}\subset \{W>0\}$. Taking the union (which is countable) over $x_0$, we finally derive that
\[
  \mathrm{LocExt}^c \ \subset\ \{W>0\},
\]
up to a null-measure set.
This concludes the proof of Theorem~\ref{thm:ks_condition}.
Note that survival of a given type almost surely upon non extinction of the process is stated as a sufficient criterion stated in~\cite{ligonniere_kesten_2024}  (yet within a different setting with varying environment and stronger moment assumptions) to recover the dichotomy between $\{W=0\}$ and the event of survival.

\section{Proof of Theorem~\ref{thm:type_convergence}}\label{sec:proof_type_convergence}

In this part, we will use several times a version of Scheffé's lemma for nonnegative random variables, which we state in our specific setting, for the sake of clarity.
\begin{lem}\label{lem:special-Scheffe}
  For $n\geq 1$ and $x\in \mathcal{X}$, let $U_{n,x}$ and $V_x$ be nonnegative random variables.
  If for all $x\in \mathcal{X}$ and $\varepsilon>0$,
  \[
   \lim_{n\to\infty} \mathbb{P}(U_{n,x} \geq V_x-\varepsilon)\ =\ 1 \quad \text{and}\quad \lim_{n\to\infty} \sum_{x\in \mathcal{X}} \mathbb{E}[U_{n,x}]\ =\ \sum_{x\in \mathcal{X}}\mathbb{E}[V_x]\ <\ \infty,
  \]
  then we have $\lim_{n\to\infty}\mathbb{E}\big[\sum_x \abs{U_{n,x} - V_x}\big]= 0$.
  Note that the first of the two conditions above is implied by the condition that $\liminf_{n\to\infty} U_{n,x} \geq V_x$ for all $x\in \mathcal{X}$.
\end{lem}
\begin{proof}
  First let us fix $x\in \mathcal{X}$ and show that $\mathbb{E}[U_{n,x} \wedge V_x]\to \mathbb{E}[V_x]$ as $n\to\infty$.
  We compute, for any $\varepsilon>0$:
  \begin{align*}
    \mathbb{E}[V_x]\ &\geq\ \mathbb{E}\big[U_{n,x} \wedge V_x\big]\\
     &\geq\ \mathbb{E}\Big[\mathds{1}_{\{U_{n,x}\geq V_x-\varepsilon\}}(V_x-\varepsilon)\Big]\\
     & =\ \mathbb{E}[V_x]-\varepsilon - \mathbb{E}\Big[\mathds{1}_{\{U_{n,x}< V_x-\varepsilon\}}(V_x-\varepsilon)\Big].
  \end{align*}
  Since $V_x-\varepsilon$ is integrable and $\mathbb{P}(U_{n,x}< V_x-\varepsilon)\to 0$, the last term above tends to $0$, and so $\liminf_{n} \mathbb{E}[U_{n,x} \wedge V_x] \geq \mathbb{E}[V_x]-\varepsilon$.
  Since $\varepsilon$ is arbitrary, we get $\mathbb{E}[U_{n,x} \wedge V_x]\to \mathbb{E}[V_x]$.
  Now, by dominated convergence we have $\sum_x \mathbb{E}[U_{n,x} \wedge V_x] \to C \coloneqq \sum_{x\in \mathcal{X}}\mathbb{E}[V_x]$, and we now just use the fact that $\abs{U_{n,x} - V_x} = U_{n,x}+V_x-2(U_{n,x} \wedge V_x)$ to get
  \[
  \lim_{n\to\infty}  \mathbb{E}\Big[\sum_{x\in \mathcal{X}} \abs{U_{n,x} - V_x}\Big] \ =\ C+C-2C\ =\ 0,
  \]
  concluding the proof.
\end{proof}

\subsection{Proof of (i) and (ii)}

 Proposition~\ref{prop:perron_frobenius} directly gives the first part of Theorem~\ref{thm:type_convergence} in the null recurrent setting. From now on we can assume $M$ is positive recurrent with $\langle\tilde{h},h\rangle =1$.

\subsubsection*{Ergodicity in mean} This paragraph is dedicated to proving convergence of the type vector in $L^1(\mathbb{B}_h)$, which corresponds to the second part (ii) in Theorem~\ref{thm:type_convergence}.
The proof relies on two intermediate lemmas. First and foremost, we recall that   
\[
\mathbb{E}\Big[\norm[\big]{\varrho^{-n} Z(n) - W\tilde{h}}_h\Big]\ =\
\mathbb{E}\Big[\sum_{x\in \mathcal{X}} \abs[\big]{\varrho^{-n}Z_x(n) - W\tilde{h}_x} h_x \Big].
\]

\begin{lem}[Restriction to a finite set of types]\label{lem:restriction_finite}
    For all $\varepsilon,\eta>0$, there exists a finite subset
    $\mathcal{X}_0\subset \mathcal{X}$ and $n_0\geq 1$ such that
    for all $n\geq n_0$,
    \begin{equation*}
        \mathbb{P}\big(\varrho^{-n}\langle Z(n), h|_{\mathcal{X}_0^c}\rangle\leq \varepsilon\big)\ \geq\ 1-\eta.
    \end{equation*}
\end{lem}

\begin{proof}
  Recall that under Assumption~\ref{ass:recurrence}, by Proposition~\ref{prop:perron_frobenius}, for all $x,y\in \mathcal{X}$, we have $\varrho^{-n}(M^n)_{xy} \to h_x\tilde{h}_y$ as $n\to\infty$.
  Therefore, using Fatou's lemma, for all $y\in \mathcal{X}$,
  \[
    \liminf_{n\to\infty} \mathbb{E}\big[\varrho^{-n}Z_y(n)\big]\ =\ \liminf_{n\to\infty} \sum_{x\in \mathcal{X}}Z_x(0)\varrho^{-n}(M^{n})_{xy}\ \geq\ \sum_{x\in \mathcal{X}}Z_x(0)h_x\tilde{h}_y\ =\ \langle Z(0), h\rangle \tilde{h}_y.
  \]
  By the martingale property and the condition $\langle \tilde{h},h\rangle=1$, we also have
  \begin{equation*}
    \sum_{x\in \mathcal{X}} \mathbb{E}\big[\varrho^{-n}Z_x(n)\big]h_x\ =\ \mathbb{E}[W_n]\ =\ \langle Z(0), h\rangle \langle \tilde{h}, h\rangle\ =\ \norm[\big]{\langle Z(0), h\rangle\cdot \tilde{h}}_h.
  \end{equation*}
  Therefore, applying Lemma~\ref{lem:special-Scheffe} with the deterministic sequence $U_{n,x}\ =\ \mathbb{E}[\varrho^{-n}Z_x(n)]h_x$ and $V_x=\langle Z(0),h\rangle  \tilde{h}_xh_x$, we get
  \begin{equation}\label{eq:cv-EZn-L1}  
   \lim_{n\to\infty} \sum_{x\in \mathcal{X}} \abs[\Big]{\mathbb{E}\big[\varrho^{-n}Z_x(n)\big] - \langle Z(0), h\rangle\cdot \tilde{h}_x}h_x\ =\ 0,
  \end{equation}
  in other words we have the convergence $\mathbb{E}[\varrho^{-n}Z(n)]\to \langle Z(0), h\rangle\cdot\tilde{h}$  in $\mathbb{B}_h$ as $n\to\infty$.
  Now let $\varepsilon, \eta >0$ and $\mathcal{X}_0$ a finite subset of $\mathcal{X}$  such that $\langle Z(0), h|_{\mathcal{X}_0^c}\rangle\leq \varepsilon\eta/2$.
  Using~\eqref{eq:cv-EZn-L1} and the fact that $u\in \mathbb{B}_h \mapsto \langle u, h|_{\mathcal{X}_0^c}\rangle\in \mathbb{R}$ is continuous, we can find $n_0$ such that
\begin{equation*}
  \forall n\geq n_0,\ \mathbb{E}\big[\varrho^{-n}\langle Z(n), h|_{\mathcal{X}_0^c}\rangle\big]\ \leq\ \varepsilon\eta. 
\end{equation*}
  We finally get with Markov's inequality that 
  \begin{equation*}
    \mathbb{P}\big(\varrho^{-n}\langle Z(n), h|_{\mathcal{X}_0^c}\rangle\geq \varepsilon\big)\ \leq\ \eta,
  \end{equation*}
  which concludes the lemma.
\end{proof}

    \begin{lem}[Truncating the offspring distribution]\label{lem:cutoff}
     Let $\varrho_1,\varrho_2$ satisfying
     \begin{equation*}
      \varrho_1\in\ (\sqrt{\varrho},\varrho),\quad 
     1< \varrho_2< \varrho\quad\text{and}\quad
      \varrho_1^2/\varrho_2>\varrho.
     \end{equation*}
      Note that since $\varrho>1$, such $\varrho_1$ and $\varrho_2$ can indeed be found. For $n\geq 1$, introduce the capped mean matrix $\bar{M}(n)$ such that for every $x,y\in \mathcal{X}$,
     \begin{equation}\label{eq:matrix_cutoff}
        \bar{M}_{xy}(n)\ \coloneqq\ \mathbb{E}\big[L^{(x)}_y \wedge (\varrho_2^n h_y\tilde{h}^2_y)\big].
     \end{equation}
    Then almost surely, there exists $n_0\in \mathbb{N}$ such that for all $n\geq n_0$ and $y\in\mathcal{X}$,
\begin{equation*}
    Z_y(n+1) \ \geq\ \big(Z(n)^\transpose \bar{M}(n)\big)_y-\varrho_1^n\tilde{h}_y.
\end{equation*}
    \end{lem}

\begin{proof}
  From the definition of the capped matrix, $\forall x,y$, $\bar{M}_{xy}(n)\leq M_{xy}$ and converges to $M_{xy}$ as $n\to\infty$.
  Let $y\in\mathcal{X}$ and $n\geq 0$ be fixed.
  By the branching property on $Z$, we can write
   \begin{align*}
    Z_{y}(n+1)\ &=\ \sum_{x\in\mathcal{X}}\sum_{i=1}^{Z_x(n)}L^{(x,i)}_y\\
&\geq\ \sum_{x\in\mathcal{X}}\sum_{i=1}^{Z_x(n)}\bar{L}^{(x,i)}_y, \quad \text{with }\bar{L}^{(x,i)}_y\ \coloneqq\ L^{(x,i)}_y \wedge (\varrho_2^n h_y\tilde{h}^2_y),
   \end{align*}
where the $(L^{(x,i)})_{i\geq 1,x\in \mathcal{X}}$ are independent, with each $L^{(x,i)}$ distributed as $L^{(x)}$. The $(\bar{L}^{(x,i)}_y)_{i\geq 1}$ have finite variance, such that 
\[
\mathrm{Var}\big(\bar{L}^{(x,i)}_y\big)\ \leq\ \mathbb{E}\big[\big(\bar{L}^{(x,i)}_y\big)^2\big]\ \leq\ \mathbb{E}\big[L^{(x,i)}_y\big]\varrho_2^n h_y\tilde{h}^2_y\ =\ M_{xy}\varrho_2^n h_y\tilde{h}^2_y.
\]
Now we can apply the Bienaymé-Chebyshev inequality to get
\begin{align}\nonumber
  \mathbb{P}\big(Z_y(n+1)\leq (Z(n)^\transpose \bar{M}(n))_y - \varrho_1^n\tilde{h}_y\big)
 \ & \leq\ \mathbb{P}\Big(\sum_{x\in\mathcal{X}}\sum_{i=1}^{Z_x(n)}\bar{L}^{(x,i)}_y\leq (Z(n)^\transpose \bar{M}(n))_y - \varrho_1^n\tilde{h}_y\Big) \\\nonumber
  &\leq\ \mathbb{E}\bigg[ \frac{1}{\varrho_1^{2n}\tilde{h}_y^2} \mathrm{Var}\Big(\sum_{x\in\mathcal{X}}\sum_{i=1}^{Z_x(n)}\bar{L}^{(x,i)}_y \;\big|\; \mathcal{F}_n \Big) \bigg]\\ \label{eq:lemma17_1}
  &\leq\ \sum_{x\in\mathcal{X} }\mathbb{E}\big[Z_x(n)\big]\frac{\mathrm{Var}(\bar{L}^{(x)}_y)}{\varrho_1^{2n}\tilde{h}^2_y}\\\nonumber
  &\leq\ \sum_{x\in\mathcal{X} }\mathbb{E}\big[Z_x(n)\big]M_{xy}h_y\Big(\frac{\varrho_2}{\varrho_1^2}\Big)^n.
\end{align}
By the union bound and harmonicity of $h$,
\begin{align*}
    \mathbb{P}\big(\exists y\in \mathcal{X},\ Z_y(n+1)\leq (Z(n)^\transpose \bar{M}(n))_y - \varrho_1^n\tilde{h}_y\big)\ &\leq\ \Big(\frac{\varrho_2}{\varrho_1^2}\Big)^n\sum_{x,y\in \mathcal{X}}\mathbb{E}\big[Z_x(n)\big]M_{xy}h_y\\
    &\leq\ \Big(\frac{\varrho_2\varrho}{\varrho_1^2}\Big)^n \varrho \cdot \langle Z(0), h\rangle.
\end{align*}
Since we picked $\varrho_1$ and $\varrho_2$ to have $\varrho_2\varrho/\varrho_1^2<1$, the previous bound is summable over $n$ and the result holds by Borel-Cantelli's lemma.
\end{proof}

Now let us prove point (ii) in Theorem~\ref{thm:type_convergence}. Take $\mathcal{X}_0$ as in Lemma~\ref{lem:restriction_finite} and fix $y\in \mathcal{X}$ and $\varepsilon>0$ for good.
Using Proposition~\ref{prop:perron_frobenius}, one can choose some $k\in\mathbb{N}^*$ such that the following holds:
\begin{equation}\label{eq:power_close_to_PF}
  \forall x\in \mathcal{X}_0,\ M^k_{xy}\ \geq\ (1-\varepsilon)\varrho^k h_x\tilde{h}_y.
\end{equation}
Using that $\tilde{h}$ is the left eigenvector associated with $\varrho$, Lemma~\ref{lem:cutoff} implies that almost surely, for all $n$ large enough, we have coordinate-wise that
\begin{align*}
    Z(n+k)^\transpose \ &\geq\ Z(n)^\transpose \bar{M}(n,k)-(\varrho^{k-1}+\varrho^{k-2}\varrho_1+\cdots + \varrho\varrho_1^{k-2} + \varrho_1^{k-1})\varrho_1^n\tilde{h}^\transpose\\
   & =\  Z(n)^\transpose \bar{M}(n,k)-(\varrho-\varrho_1)^{-1}(\varrho^k-\varrho_1^k)\varrho_1^n\tilde{h}^\transpose, \\
   & \geq\ Z(n)^\transpose \bar{M}(n,k)-(\varrho-\varrho_1)^{-1}\varrho^k\varrho_1^n\tilde{h}^\transpose,
\end{align*}
where $\bar{M}(n,k)\coloneqq\bar{M}(n)\bar{M}(n+1)\cdots \bar{M}(n+k-1)$.
Because $\bar{M}(n)$ increase coordinate-wise to $M$ as $n\to\infty$, so does $\bar{M}(n,k)$ to $M^k$.
We now fix $\eta>0$, and we may pick $n_0$ large enough to ensure that for all $n\geq n_0$,
\begin{equation*}
    \left\{ 
        \begin{array}{ll r}(\mathrm{a})&
       \mathbb{P}\big(\varrho^{- n}\langle Z(n),h|_{\mathcal{X}_0^c}\rangle \leq \varepsilon\big)\ \geq\ 1-\eta  & (\text{by Lemma~\ref{lem:restriction_finite}});\\[5mm]
        (\mathrm{b})&
       \forall x\in \mathcal{X}_0,\ 
    \bar{M}_{xy}(n,k)\ \geq\ (1-2\varepsilon)\varrho^kh_x\tilde{h}_y & (\text{$\mathcal{X}_0$ finite and using~\eqref{eq:power_close_to_PF}});\\[5mm]
     (\mathrm{c}) &
    \begin{multlined}[t]\mathbb{P}\big(\forall n\geq n_0,\ Z(n+k)^\transpose > Z(n)^\transpose \bar{M}(n,k) - (\varrho-\varrho_1)^{-1}\varrho^k\varrho_1^n\tilde{h}^\transpose \big)\\[-2mm] \geq\ 1-\eta\end{multlined} \hspace{-7mm} & (\text{by Lemma~\ref{lem:cutoff}});\\[5mm]
    (\mathrm{d})&
  \displaystyle (\varrho-\varrho_1)^{-1}\Big(\frac{\varrho_1}{\varrho}\Big)^{n}\ <\ \varepsilon  & (\text{since $\varrho_1<\varrho$});\\[5mm]
   (\mathrm{e})&
\mathbb{P}\big(\forall n \geq n_0,\ \abs{\varrho^{-n}\langle Z(n), h\rangle-W}\leq \varepsilon\big)\ \geq\ 1-\eta & (\text{a.s. convergence of $W_n$}\big).
\end{array}\right.
\end{equation*}
Let $n\geq n_0$. By $(\mathrm{d})$, we have
\begin{multline*}
\big\{ Z(n+k)^\transpose \ \geq\ Z(n)^\transpose\bar{M}(n,k)-(\varrho-\varrho_1)^{-1}\varrho^k\varrho_1^n\tilde{h}^\transpose\big\} \\\subset\ \Big\{ \frac{Z(n+k)^\transpose}{\varrho^{(n+k)}}\ \geq\ \varrho^{-(n+k)}Z(n)^\transpose\bar{M}(n,k)-\varepsilon \tilde{h}^\transpose \Big\}.
\end{multline*}
Thus using $(\mathrm{c})$, the probability of the latter event is at least $1-\eta$.
Furthermore by $(\mathrm{b})$, we have $(Z(n)^\transpose \bar{M}(n,k))_y \geq \varrho^k (1-2\varepsilon)\big(\sum_{x\in \mathcal{X}_0} Z_x(n)h_x\big)\tilde{h}_y = \varrho^k (1-2\varepsilon)\langle Z(n),h|_{\mathcal{X}_0}\rangle  \tilde{h}_y$.
Whence we derive that
\begin{equation*}
    \mathbb{P}\Big( \frac{Z_y(n+k)}{\varrho^{n+k}}\ \geq\ \big( (1-2\varepsilon)\varrho^{-n}\langle Z(n),h|_{\mathcal{X}_0}\rangle-\varepsilon\big) \tilde{h}_y \Big)\ \geq\ 1-\eta.
\end{equation*} 
Now, by $(\mathrm{a})$ and $(\mathrm{e})$, we have $\mathbb{P}\big(\abs{\varrho^{-n}\langle Z(n),h|_{\mathcal{X}_0}\rangle-W}\leq 2\varepsilon \big)\geq 1- 2\eta$.
It follows that
\[
  \mathbb{P}\big( \varrho^{-(n+k)}Z_y(n+k)\ \geq\ ( (1-2\varepsilon)(W-2\varepsilon)-\varepsilon) \tilde{h}_y \big)\ \geq\ 1-3\eta,
\]
and since $\varepsilon$ and $\eta$ are arbitrary, standard arguments yield:
\begin{equation*}
   \forall \varepsilon > 0, \qquad \mathbb{P}\big(\varrho^{-n}Z_y(n)\geq W\tilde{h}_y - \varepsilon\big)\ \underset{n\to\infty}{\longrightarrow}\ 1.
\end{equation*}

Combined with the fact that $\mathbb{E}\big[\sum_{y}\varrho^{-n}Z_y(n)h_y\big] = \mathbb{E}[W_n]=\mathbb{E}[W] = \mathbb{E}[\sum_{y}W\tilde{h}_yh_y]$, applying Lemma~\ref{lem:special-Scheffe} to $U_{n,x}=\varrho^{-n}Z_x(n)h_x$ and $V_x=W\tilde{h}_xh_x$ we get that
\[
 \lim_{n\to\infty} \mathbb{E}\bigg[\sum_{x\in\mathcal{X}}\abs[\Big]{\frac{Z_x(n)}{\varrho^n} - W\tilde{h}_x}h_x \bigg] \ =\ 0,
\]
which concludes the proof of (ii).

\subsection{Proof of (iii)}
\subsubsection*{(iii)-\ref{ass:variance_condition}: Offspring distributions with finite variance}
Assume $\forall x,y\in \mathcal{X}$, $\mathrm{Var}(L_y^{(x)})<\infty$ and satisfy
\begin{equation*}
  \sum_{x,y\in \mathcal{X}}\tilde{h}_x \mathrm{Var}(L_y^{(x)}) \tilde{h}_y^{-2}\ <\ \infty.
\end{equation*}
Then, we can improve Lemma~\ref{lem:cutoff} using the mean matrix directly in place of the capped matrix $\bar{M}$:
following the same steps as in the proof of
Lemma~\ref{lem:cutoff} but without replacing $L^{(x)}_y$ by $\bar{L}^{(x)}_y$, we obtain the analogue of~\eqref{eq:lemma17_1}, that is
\[
\mathbb{P}\big(Z_y(n+1)\leq (Z(n)^\transpose M)_y - \varrho_1^n\tilde{h}_y\big) \; \leq \; \sum_{x\in\mathcal{X} }\mathbb{E}\big[Z_x(n)\big]\frac{\mathrm{Var}(L^{(x)}_y)}{\varrho_1^{2n}\tilde{h}^2_y},
\]
so summing over $y\in \mathcal{X}$ and using a union bound, we get
\begin{align*}
  \mathbb{P}\big(\exists y\in\mathcal{X},\ Z_y(n+1)\leq (Z(n)^\transpose M)_y - \varrho_1^n\tilde{h}_y\big)
 \ & \leq\ \sum_{x,y\in \mathcal{X}} \mathbb{E}[Z_x(n)] \frac{\mathrm{Var}(L_{y}^{(x)})}{\rho_1^{2n}\tilde{h}_y^2} \\
 &\leq\ \Big(\frac{\rho}{\rho_1^2}\Big)^n \langle Z(0),h\rangle \sum_{x,y\in \mathcal{X}} \tilde{h}_x\mathrm{Var}(L_{y}^{(x)})\tilde{h}_y^{- 2}.
\end{align*}
By assumption, the last sum is a finite constant, so summing over $n$ and using $\varrho_1>\sqrt{\varrho}$, the Borel--Cantelli lemma applies and we obtain that almost surely, for all $n$ large enough and all
$y\in \mathcal{X}$,
\[
  Z_y(n+1)\ \geq\ (Z(n)^\transpose M)_y - \varrho^n_1\tilde{h}_y,\quad \forall k\in\mathbb{N}.
\]
Multiplying on the right by $M$ iteratively, an easy induction shows that almost surely, for all $n$ large enough, all $k\geq 1$ and all
$y\in \mathcal{X}$,
\begin{align*}
  Z_y(n+k)\ &\geq\ (Z(n)^\transpose M^k)_y - \sum_{i=1}^k \varrho^{i-1} \varrho_1^{n+k-i} \tilde{h}_y\\
  &\geq\ (Z(n)^\transpose M^k)_y - (\varrho-\varrho_1)^{-1}\varrho^k\varrho^n_1\tilde{h}_y.
\end{align*}
This implies that for $n$ large enough, we have for any $y\in \mathcal{X}$,
\[
\frac{Z_y(n+k)}{\varrho^{n+k}}\ \geq\ \Big(\frac{Z(n)^\transpose}{\varrho^n} \frac{M^k}{\varrho^k}\Big)_y- (\varrho-\varrho_1)^{-1}\Big(\frac{\varrho_1}{\varrho}\Big)^n\tilde{h}_y,\quad \forall k\in\mathbb{N}.
\]
Because $\varrho^{-k}M^k\to h\tilde{h}^\transpose$ and $\varrho_1<\varrho$, it is not hard to see that for any $v\in \mathbb{B}_h$, we have $\liminf_k\varrho^{-k}(v^\transpose M^k)_y \geq (v^\transpose  h \tilde h^\transpose )_y = \langle v,h\rangle  \tilde{h}_y$.
So by letting $k\to \infty$, we obtain that for $n$ large enough, 
\[
  \liminf_{m\to \infty}\frac{Z_y(m)}{\varrho^{m}}\ \geq\  \frac{\langle Z(n),h\rangle}{\varrho^n} \tilde{h}_y\ =\ W_n\tilde{h}_y,
\]
and since the left-hand side does not depend on $n$, we derive simply letting $n\to\infty$ that $\liminf_{m}\varrho^{-m}Z_y(m) \geq W\tilde{h}_y$. The conclusion holds again by applying Scheffé's lemma for nonnegative functions, using the fact that $\sum_yW\tilde{h}_yh_y = W = \lim_n \sum_y \varrho^{-n}Z_y(n)h_y$.

\subsubsection*{(iii)-\ref{ass:weaker_entropy_condition}: Entropy condition}
The key will again be to show that for all $y\in \mathcal{X}$,
\begin{equation}\label{eq:asCV-goal}
  \liminf_{m\to \infty}\varrho^{-m}Z_y(m)\ \geq\ W\tilde{h}_y,
\end{equation}
and conclude as above.
Let us first justify that we can assume without loss of generality that the initial population $Z(0)$ is finite.
In the general case, note that for any $\mathcal{X}_0\subset \mathcal{X}$, we can decompose $Z(n)=Z^{\mathcal{X}_0}(n)+Z^{\mathcal{X}_0^c}(n)$, where $(Z^{\mathcal{X}_0}(n))_{n\geq 0}$ and $(Z^{\mathcal{X}_0^c}(n))_{n\geq 0}$ are versions of the branching process started respectively from $Z(0)|_{\mathcal{X}_0}$ and $Z(0)|_{\mathcal{X}^c_0}$.
Note that we can similarly write $W=W^{\mathcal{X}_0}+W^{\mathcal{X}_0^c}$ with obvious definitions, and it is not hard to see that $W^{\mathcal{X}_0^c}\to 0$ as $\mathcal{X}_0\uparrow \mathcal{X}$.
Therefore, if~\eqref{eq:asCV-goal} is shown to hold for any finite initial population, then in the general case we have that for all $y\in\mathcal{X}$, $\liminf_{m}\varrho^{-m}Z_y(m) \geq \sup_{\mathcal{X}_0}W^{\mathcal{X}_0}\tilde{h}_y = W\tilde{h}_y$, where the supremum is taken on finite $\mathcal{X}_0\subset \mathcal{X}$, and the proof will be complete.
From now on, we therefore assume that $Z(0)$ is finite.

Recall the capped matrix $\bar{M}(n)$ from~\eqref{eq:matrix_cutoff}, defined for each $n\in \mathbb{N}$.
From Lemma~\ref{lem:cutoff}, we have for all $y\in\mathcal{X}$ that $Z(n+1)_y \geq (Z(n)^\transpose \bar{M}(n))_y - \varrho_1^n\tilde{h}_y$ almost surely for all $n$ large enough, which we can rewrite as (note that each of the following terms is nonnegative)
\[
  \big(Z(n)^\transpose M\big)_y\ \leq\ Z_y(n+1) + \big(Z(n)^\transpose \big(M-\bar{M}(n)\big)\big)_y + \varrho_1^n\tilde{h}_y.
\]
Iterating the inequality above by multiplying by $M$ on the right, we obtain that for $n$ large enough, for all $y\in \mathcal{X}$ and for all $k\in \mathbb{N}$,
\begin{multline}\label{eq:ineq_iii_c}
  \frac{(Z(n)^\transpose M^k)_y}{\varrho^{n+k}}\ \leq\  \frac{Z_y(n+k)}{\varrho^{n+k}}\\ +\ \sum_{i=0}^{k-1}\bigg(\frac{Z(n+i)^\transpose}{\varrho^{n+i}}(M-\bar{M}(n+i))\frac{M^{k-i-1}}{\varrho^{k-i}}\bigg)_y + (\varrho-\varrho_1)^{-1}\Big(\frac{\varrho_1}{\varrho}\Big)^n\tilde{h}_y.
\end{multline}
We will first take the limit $k\to \infty$ and then $n\to \infty$ in this inequality to obtain $\liminf_{m}\varrho^{-m}Z_y(m) \geq W\tilde{h}_y$ for all $y\in \mathcal{X}$.
As $\varrho_1<\varrho$, the term $(\varrho-\varrho_1)^{-1}(\varrho_1/\varrho)^n\tilde{h}_y$ disappears in the second limit $n\to\infty$, so the delicate term is the sum $R_{n,k}\coloneqq \sum_{i=0}^{k-1}\frac{Z(n+i)^\transpose}{\varrho^{n+i}}\big(M-\bar{M}(n+i)\big)\frac{M^{k-i-1}}{\varrho^{k-i}}$.
First, note that
\[
   \norm{R_{n,k}}_h\ =\ \langle R_{n,k},h\rangle \ =\ \frac{1}{\varrho}\sum_{i=0}^{k-1} \frac{Z(n+i)^\transpose}{\varrho^{n+i}}\big(M-\bar{M}(n+i)\big) h,
\]
which is non-decreasing in $k$, and tends to
\[
  \lim_{k\to \infty}\norm{R_{n,k}}_h \ \eqqcolon\ \mathcal{R}_n \ =\ \frac{1}{\varrho}\sum_{i\geq n} \frac{Z(i)^\transpose}{\varrho^{i}}\big(M-\bar{M}(i)\big) h.
\]
Now $(\mathcal{R}_n)$ is nonincreasing, so to show that it converges almost surely to $0$ it suffices to show that $\mathbb{E}[\mathcal{R}_n]\to 0$.
Since we have assumed that $Z(0)$ has a finite number of positive coordinates, there exists $\alpha > 0$ such that $Z(0) \leq \alpha \tilde{h}$ holds coordinate-wise, which implies that $\mathbb{E}[Z_y(i)] =(Z(0)^\transpose M^i)_y\leq \alpha \varrho^i \tilde{h}_y$ for all $y\in \mathcal{X}$.
Using this, we compute
\begin{align}\nonumber
  \mathbb{E}\big[ \mathcal{R}_n\big]\ &=\ \frac{1}{\varrho}\sum_{i\geq n}\mathbb{E}\Big[\frac{Z(i)^\transpose}{\varrho^{i}}\big(M-\bar{M}(i)\big)h \Big] \\
  &\leq\  \frac{\alpha}{\varrho}\sum_{i\geq n} \tilde{h}^\transpose\big(M-\bar{M}(i)\big)h. \label{eq:rest_cv_series}
\end{align}
Let us show that this is the remainder of a convergent series, \textit{i.e.} that the following converges:
\begin{align*}
  \sum_{i\geq 1} \tilde{h}^\transpose\big(M-\bar{M}(i)\big)h\ &=\ \sum_{x,y\in \mathcal{X}}\tilde{h}_x \Big( \sum_{i\geq 1} \big(M-\bar{M}(i)\big)_{xy} \Big) h_y. \\
  &=\ \sum_{x,y\in \mathcal{X}}\tilde{h}_x\mathbb{E}\Big[ \sum_{i\geq 1} \left( L^{(x)}_y - \varrho_2^i h_y\tilde{h}_y^2 \right)_+\Big]h_y. 
\end{align*}
Now notice that for any $a,b>0$,
\[
\sum_{i\geq 1}\left(a - \varrho_2^i b \right)_+ \ \leq\ a \abs[\big]{\{i \in \mathbb{N} : \varrho_2^ib \leq a\}}\ =\ a\floor[\Big]{\frac{\log_+(a/b)}{\log \varrho_2}} \ \leq\ a\frac{\log_+(a/b)}{\log \varrho_2}.
\]
Therefore, we obtain
\[
  \sum_{i\geq 1} \tilde{h}^\transpose \big(M-\bar{M}(i)\big)h \ \leq\ \frac{1}{\log \varrho_2} \sum_{x,y\in \mathcal{X}} \tilde{h}_x \mathbb{E}\Big[L^{(x)}_y\log_+\big( L^{(x)}_y h_y^{-1}\tilde{h}_y^{-2} \big) \Big] h_y.
\]
It results that Condition~(iii)-\ref{ass:weaker_entropy_condition} in Theorem~\ref{thm:type_convergence} implies the convergence of the series $\sum_{i\geq 1} \tilde{h}^\transpose \big(M-\bar{M}(i)\big)h$, hence the fact that $\mathcal{R}_n\to 0$ almost surely.

This means $\lim_{n\to\infty}\lim_{k\to \infty}\norm{R_{n,k}}_h = 0$ almost surely, and so we can take the limits in~\eqref{eq:ineq_iii_c} to establish $\liminf_n \varrho^{-n}Z(n) \geq W\tilde{h}$, and conclude as in the proof of~(iii)-\ref{ass:variance_condition}.

It remains to retrieve~(iii)-\ref{ass:weaker_entropy_condition} from the conditions stated in the last part of Theorem~\ref{thm:type_convergence}, namely: we assume Condition~\eqref{eq:simple_condition} and $-\sum_{x\in \mathcal{X}} \tilde{h}_xh_x \log\big( \tilde{h}_xh_x\big) < \infty$.
Recalling from $\langle \tilde{h}, h\rangle =1$ that $-\log(\tilde{h}_yh_y) \geq 0$ for all $y\in \mathcal{X}$, we compute:
\begin{align*}
\sum_{x,y\in \mathcal{X}} \tilde{h}_x \mathbb{E}\Big[L^{(x)}_y&\log_+\big( L^{(x)}_y h_y^{-1}\tilde{h}_y^{-2} \big) \Big] h_y \\
&\leq\ 
\sum_{x,y\in \mathcal{X
}} \tilde{h}_x \mathbb{E}\Big[L^{(x)}_y\log_+\big( L^{(x)}_yh_y \big) \Big] h_y
-2 \sum_{x,y\in \mathcal{X}} \tilde{h}_x \mathbb{E}\Big[L^{(x)}_y \log\big( \tilde{h}_yh_y \big) \Big] h_y \\
&\leq\ 
\sum_{x\in \mathcal{X}} \tilde{h}_x \mathbb{E}\Big[\big\langle L^{(x)},h\big\rangle\log_+\big\langle L^{(x)},h\big\rangle \Big]
- 2 \varrho \sum_{y\in \mathcal{X}} \tilde{h}_yh_y \log\big( \tilde{h}_yh_y\big),
\end{align*}
where we have used the equality $\sum_x \tilde{h}_x \mathbb{E}[L^{(x)}_y] = (\tilde{h}^\transpose M)_y = \varrho \tilde{h}_y$.
This completes the proof of Theorem~\ref{thm:type_convergence}.

\section{Application to general models of epidemics}\label{sec:applications}

This final section is dedicated to the study of large-population
limits of some general discrete-time epidemics processes.

\subsection*{Motivation and related work}
Individuals are often not equally infectious throughout the course of
an infection. Viral load, symptom severity, and behavioral changes
vary over time and lead to different transmission potentials at
different stages of the disease. Likewise, susceptibility to infection
may depend on the stage of infection of the transmitting individual,
reflecting differences in pathogen shedding, contact intensity, or
immune response. Incorporating such stage-dependent heterogeneity
allows the model to capture biological features. Thus, it is standard
in epidemic modeling to assume that the infectivity of an individual
depends on how long that individual has been infected, such that the
model is structured in age. We aim to describe the onset of an
epidemic in a non-Markov, varying-susceptibility epidemic model.
 
First consider a population of size $N\in \mathbb{N}$ where
individuals are either \emph{susceptible} or \emph{infected}. The
epidemics process evolves in time: any individual $i$ infected at time
$t$ carries a random point process
$\mathcal{P}_i=\sum_{k\geq 1} \delta_{t_k}$ that describes its
infectious contacts after time $t$: for each $k$, at time $t+t_k$, an
individual is uniformly chosen in the population and becomes infected
if it was susceptible up to this time. Under the assumption that the
$(\mathcal{P}_i)$ are i.i.d.\ and technical assumptions on their
distribution, \citet{BR13} showed a functional law of large numbers
for the size of the epidemics as a function of time, as $N\to \infty$,
and the limit is described by a general Kermack--McKendrick
integro-differential equation \citep{KM27}. A similar result appears
in \citet{duchamps_general_2023}, where we relaxed technical
assumptions and considered the case of a varying contact rate function
modeling the effect of public intervention on the spread of the
epidemic. We refer to~\cite{forien2022recentadvancesepidemicmodeling}
for a review of some Markovian and non-Markovian stochastic models of
epidemics. In these works, individuals remain infectious during
arbitrarily distributed periods of time, but there is no varying
susceptibility among the population. We propose a framework that
introduces some heterogeneity in the population -- consider that
individuals have different behaviors regarding contacts with infected
individuals at different stages of the infection (for instance
healthcare workers may be more regularly in contact with advanced
stages of infections, and also may spread the infection differently
than non-healthcare workers). In the following framework, our aim is
to couple the beginning of the epidemics with a multi-type branching
process, which will allow us to exploit our main results.

An important part of these works consists in understanding the
beginning of the epidemics, which is generally approximated by a
Crump--Mode--Jagers process \citep{jagers_general_2008}. However, this
is not sufficient to describe the whole trajectory of the epidemics:
the study of a backwards-in-time process of potential
``infectors'' is also important to study. Interestingly, the
backwards-in-time process in \citet{duchamps_general_2023} is more
regular than its forwards-in-time counterpart: the uniform choice of
infection implies there is additional independence when going back
through potential infectors.

In this application section, the heterogeneity we introduce in the
population will prevent this additional regularity in the
backwards-in-time process: both branching processes will be
countable-type branching processes in the large-population
approximation. We now describe the epidemics model through its
so-called infection graph $\mathcal{G}_N$, which in broad terms is the
weighted directed graph or directed network (weights are on the edges,
and we will in fact call them \emph{lengths}) with vertex set
$\{1,\dots,N\}$ that is built by drawing an edge $i\to j$ with length
$x$ if individual $i$ would infect $j$ at infection-age $x$. The
infection started from individual $1$ is then described by the
subgraph composed of all vertices that are reachable from $1$.

The infection graph $\mathcal{G}_N$ corresponding to this model is
effectively a ``layering'' of a countable number of directed
configuration models on $\{1, \dots, N\}$, where the edges have
integer-valued lengths. Shortest paths in configuration models with
edge lengths have been extensively studied \citep[see
\emph{e.g.}][]{BHH10,bhamidi_front_2014,vO18}, but never in such a
layered setting where countable-type branching processes
emerge. Indeed, in order to study the distribution of geodesics (as
$N\to\infty$) in this random graph, we need to understand the start of
the epidemics. Note that here we only focus on the branching
approximation, the study of geodesics in this model should be the
subject of future work.

\subsection*{Layered configuration model}
We thus model the infection graph of the epidemic as
follows: we endow every individual $i\in[N]$ with a pair of sequences
$((D_{i}^{+}(x))_{x\geq 1}, (D_{i}^{-}(x))_{x\geq 1})$; $D^+_i(x)$
(resp.\ $D^-_i(x)$) is the type-$x$ out-degree (resp.\ in-degree) of
vertex $i$, and independently for each $x\geq 1$ a directed
configuration model is drawn using type-$x$ in- and out-degrees. Each
edge formed by joining two type-$x$ half-edges is assigned length $x$.
In this construction, $D^+_i(x)$ represents the infectious contacts
that individual $i$ makes at infection-age $x$, and $D^-_i(x)$
represents the susceptibility of $i$ to infectious contacts by
infected individuals with infection-age $x$. Similar models in
continuous-time (but without varying susceptibility) in
\cite{duchamps_general_2023} are shown to encompass compartmental
models of epidemics.

Let us now make this construction more precise by introducing some
notation and some assumptions on the degree sequences. The directed
configuration model (DCM) is a model of random graph introduced
by~\cite{bollobas_probabilistic_1980}, whose construction is best
described algorithmically by assigning directed half-edges (in the
following, we will call them \emph{instubs} and \emph{outstubs} for
short) to vertices and matching them uniformly at random (see
Figure~\ref{fig:illustration_singletype_DCM} for an illustration). Let
us directly define our so-called layered version:
\begin{itemize}
\item Let $(D^{+},D^-)=((D^{+}(x))_{x\geq 1},(D^-(x))_{x\geq 1})=$ be
  a random variable taking values in the pairs of sequences of
  non-negative integers, and let $(D^+_i,\ D^-_i)_{i\geq 1}$ be an
  i.i.d.\ sequence of copies of $(D^{+},D^-)$.
\item For each $i\in \mathbb{N}$ and $x\geq 1$, we attach to vertex
  $i$ a number $D^+_i(x)$ of outstubs of type $x$, and a number
  $D^-_i(x)$ of instubs of type $x$.
\item For a fixed $N\geq 1$, we add a graveyard vertex denoted by
  $\dagger$, and such that for every type $x\in\mathbb{N}$ the
  $x$-degree of $\dagger$ correspond to the extra unmatched stubs.
  That is,
  \[
    D_{\dagger}^{+}(x)\ =\ \Big(\sum_{i=1}^N D^{+}_{i}(x) -
    \sum_{i=1}^N D^{-}_{i}(x)\Big)_+,
  \]
  and likewise for $D_{\dagger}^{-}(x)$. Note that we do not write the
  dependence on $N$ here to keep the notation short.
\item Independently for each $x\geq 1$, choose a uniform bipartite
  matching of all $x$-instubs with all $x$-outstubs among those of
  vertices indexed by $\{1,\dots,N, \dagger\}$. The directed
  network $\mathcal{G}_N$ is built by drawing a directed edge with
  length $x$ for each matched pair of $x$-stubs (the directed edge
  goes from the vertex attached to the outstub and towards the vertex
  attached to the matched instub).
\end{itemize}
Note that equivalently (and we will use this in the following), one
can realize this last step algorithmically as follows: pick an
arbitrary type $x$ and an arbitrary unmatched $x$-instub or
$x$-outstub; match it uniformly at random among the remaining
possibilities; repeat these steps until all stubs are matched (which
may take an infinite time as there may be countably many stubs to be
matched).
\begin{figure}[h]
    \centering
    \includegraphics[width=0.95\textwidth]{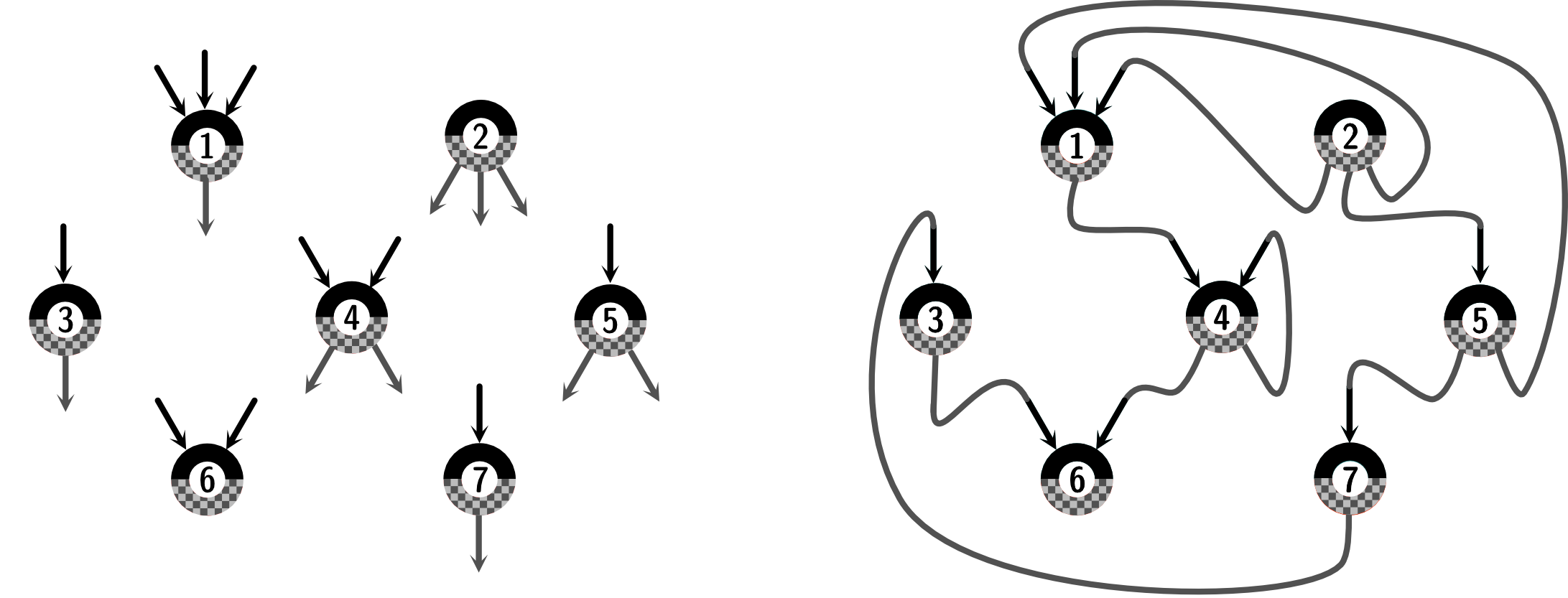}
    \caption{The DCM on $N=7$ vertices. The right panel shows a possible outcome for the matching.}
    \label{fig:illustration_singletype_DCM}
\end{figure}
Note that $\mathcal{G}_N$ may have loops and multi-edges, and also
vertices may have infinite total degrees. However, we are interested
in the lengths of directed paths in this graph: from any focal vertex
$i\in [N]$ and for any fixed radius $r\geq 1$, there are only finitely
many paths (succession of directed edges) of total length less than
$r$ in $\mathcal{G}_N$.

Recall that the directed paths in the graph $\mathcal{G}_N$ model the
potential infections that may occur in a non-Markovian, discrete-time
epidemic with heterogeneous susceptibilities in the population. The
length of an edge $i\to j$ is the time needed after the infection of
$i$ for $j$ to become infected as a result of a contact with $i$. We
are typically interested in how the infection spreads from a focal
initially infected vertex, or conversely how it may reach this focal
vertex from potential sources of infection. By exchangeability, we use the vertex
indexed by $1$ as the focal vertex. Writing $\vec{d}(i,j)$ for the
minimal length of a directed path between two vertices $i$ and $j$, we
define, for $t\geq 1$,
\[
  \mathcal{V}_N(t) \ \coloneqq\ \big\{i\in [N]:\ \vec{d}(1,i) \leq t\ \text{ or }\
  \vec{d}(i,1) \leq t\big\}
\]
the set of vertices that are at a directed distance of $1$ that is
less than $t$ in $\mathcal{G}_N$. We are interested in the rooted
decorated directed network $\mathcal{G}_N|_{t}$ whose vertex set is
$\mathcal{V}_N(t)$, with a root indexed by $1$, whose edge set is that
induced by $\mathcal{G}_N$, and where vertices are additionally
decorated with their total number of in- and outstubs of all types.
More precisely, we consider rooted decorated directed networks up to
isomorphism, \textit{i.e.} we forget about the labeling of vertices in
$\mathcal{G}_N|_{t}$. We refer to Figure~\ref{fig:exploration} for an
illustration of the part of $\mathcal{G}_N|_t$ accessible from the
root.
\begin{figure}[h]
    \centering
    \includegraphics[width=0.6\textwidth]{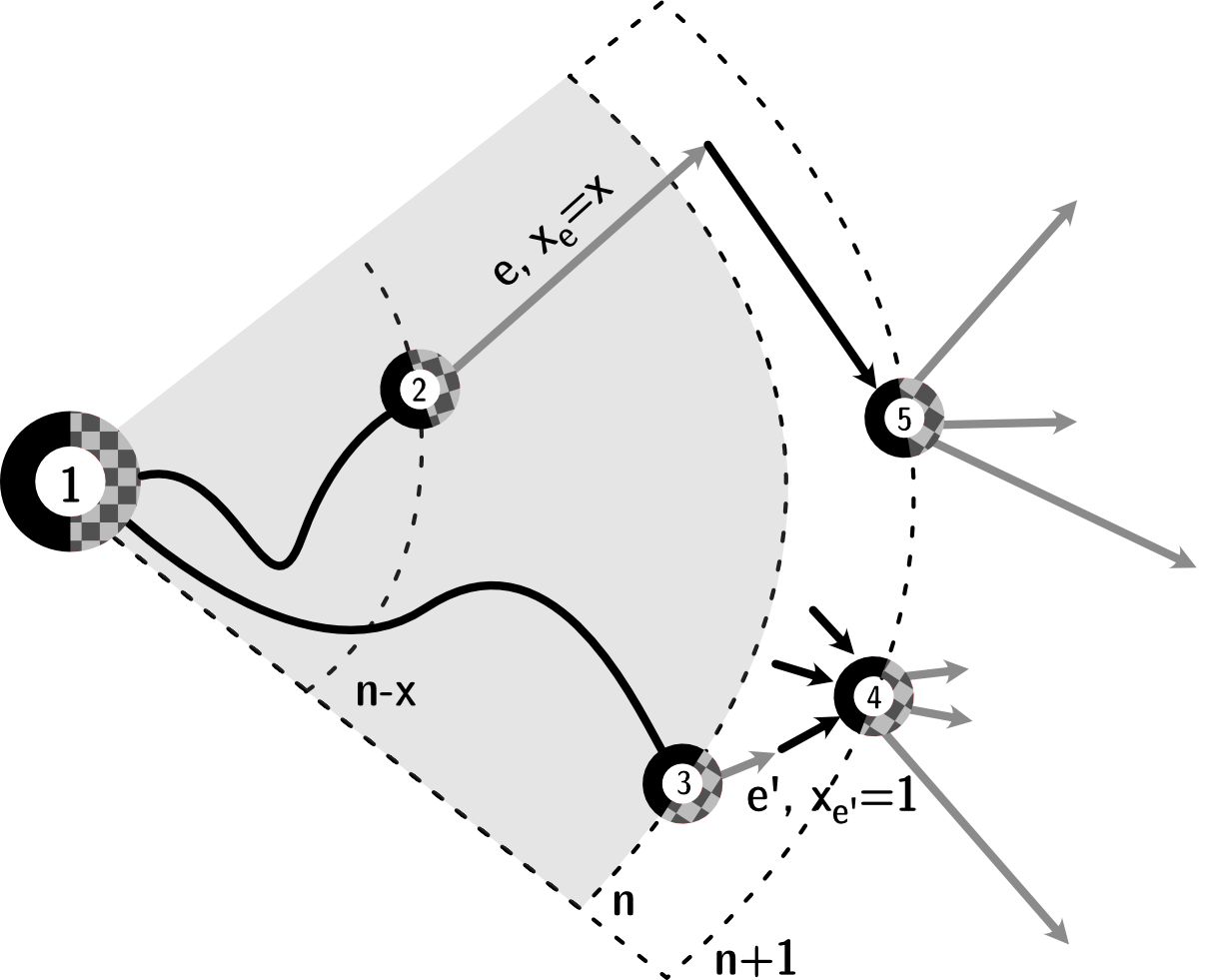}
    \caption{The out-exploration process of $\mathcal{G}$ from a
      source vertex. To find vertices at distance $n+1$ from the
      source, we traverse stubs with type $x$ that were added to the process at
      time $n-x+1$. The unmatched stubs in this picture are those going
      over the time horizon $n+1$, and those going backward from
      explored vertices.}
    \label{fig:exploration}
\end{figure}

We will be able to use Theorems~\ref{thm:ks_condition}
and~\ref{thm:type_convergence} to describe the distribution of
potential infectious contacts at time $t_N=\Theta(\log N)$ from the
focal vertex, as $N\to \infty$ and under some additional assumptions.
Among those, we impose the following moment conditions:
\begin{equation}\label{eq:moments}
  \mathbb{E}\big[(D^{\pm}_x)^2\big]\ <\ \infty,\qquad
  \mathbb{E}[D_x^-]\ =\ \mathbb{E}[D_x^+]\ =\ m_x, \qquad
  x\geq 1.
\end{equation}
In the following, we will let $\sigma_{x,\pm}^2$ denote the variance of
$D_x^{\pm}$.

\subsection*{The branching process approximation}

We will exploit our results by coupling the neighborhood
$\mathcal{G}_N|_{t}$ of the focal vertex in the infection graph
$\mathcal{G}_N$ with a similar neighborhood in a tree-like directed
network that is built from the $t$ first steps of a specific
multi-type branching process. Let us first describe the tree using a
recursive random construction:
\begin{itemize}
\item We first start with $\mathcal{T}_0$, composed of a single vertex
  (the root) labeled by $1$ and endowed with $D_1^\pm(x)$ in- and
  outstubs of type $x$, for all $x\geq 1$, where
  $(D_1^{\pm}(x))_{x\geq 1}$ is a copy of $(D^\pm(x))_{x\geq 1}$.
\item For a given network $\mathcal{T}_k$, we build
  $\mathcal{T}_{k+1}$ by ``exploring'' each type-$x$ in- and outstub,
  for each $x\geq 1$, in the following way. Replace each type-$x$
  outstub by a directed edge with length $x$ going to a new vertex $u$
  equipped with a collection of stubs randomly drawn independently
  of everything else with the distribution of $(D^\pm(x))_{x\geq 1}$
  biased by $D_x^-$, \textit{i.e.} if $L^{+}_x$ denotes such a random
  variable, then it is characterized by: for every measurable bounded
  function $f:(\mathbb{N}^{\mathbb{N}})^2 \to \mathbb{R}$,
\begin{equation}\label{eq:offspring_distrib}
  \mathbb{E}\big[f(L_x^{+})]\ =\
  \frac{1}{m_x}\mathbb{E}\big[D^-_xf((D^{\pm})_{x\geq 1})\big].
\end{equation}
One of the drawn type-$x$ instubs is used as the head of the directed
edge going to $u$. The exploration of instubs, and in particular the
definition of $L^-_x$, is done in the obvious analogous way.
\item The rooted directed random network $\mathcal{T}$ is the (unique)
  network consisting of all added vertices and edges in such a
  potentially infinite exploration procedure.
\end{itemize}
For $t\geq 1$, we define $\mathcal{T}|_t$ similarly as
$\mathcal{G}_N|_t$, by keeping only the directed paths, going from or
going to $1$, of length less than or equal to $t$.

In order to state our result, we need to use an alternative
description of an age-structured branching process linked to
$\mathcal{T}$. Let us introduce the countable state space
$\mathcal{X}= \{(x,a) : x\in \mathbb{N},\ a\leq x\}$ and focus first
on outstubs. Let us say that an edge in $\mathcal{T}$ is \emph{born}
at time $s\geq 0$ if it starts from a vertex $u$ such that
$\vec{d}(1,u)=s$. Then
$Z^+(t)=(Z^+_{x,a}(t))_{(x,a)\in \mathcal{X}}$ can be defined as
follows:
\[
  Z^+_{x,a}(t)\ =\ \text{number of edges in $\mathcal{T}$ with length $x$ born at time
  }t-a+1.
\]
Then by definition $(Z^+(t))_{t\geq 0}$ is a countable-type
Galton--Watson process where any particle of type $(x,a)$ with $a<x$
always reproduce by giving birth to a single particle of type
$(x,a+1)$, and particles of type $(x,x)$ reproduce by giving birth to
$L_{x,y}^+$ particles of type $(y,1)$ for each $y$, where
$L_x^+=(L_{x,y}^+)_{y\geq 1}$ is the biased random variable given
by~\eqref{eq:offspring_distrib}. We can see $Z^+$ as a branching
process of stubs (or potential infectious contacts) describing the
spread of the epidemic. We define $Z^-$ in the obvious analogous way.
When discussing the type $(x,a)$ of a stub, we say that $a$ is the age
of the stub.

To apply our results to the $Z^{\pm}$ processes, we must make some
minimal assumptions. Let $M^+,M^-\in \mathbb{R}^{\mathcal{X}^2}$ be
the respective mean matrices, which satisfy
\begin{equation}
  \label{eq:mean_matrices}
  \begin{aligned}
    M^\pm_{(x,a),(y,a')}\ &=\ \mathds{1}_{x=y, a'=a+1}, \qquad a < x \\
    M^+_{(x,x),(y,1)}\ &=\ \frac{1}{m_x}\mathbb{E}[D^-_x D^+_y] \\
    M^-_{(x,x),(y,1)}\ &=\ \frac{1}{m_x}\mathbb{E}[D^+_x D^-_y].
  \end{aligned}
\end{equation}
Let us assume that
\begin{equation}
  \label{eq:Mplus_satisfies_I}
  M^+\text{ satisfies Assumption~\ref{ass:recurrence} and is positive recurrent},
\end{equation}
and denote by $\rho>1$ its Perron eigenvalue, and by $\tilde{h}^+$ and
$h^+$ the associated left and right eigenvectors. Using the first line
of~\eqref{eq:mean_matrices}, we easily derive that for all
$(x,a)\in \mathcal{X}$,
\begin{equation} \label{eq:relations-hxa-hxx}
  \tilde{h}^+_{x,a}\ =\ \rho^{x-a}\tilde h^+_{x,x} \quad \text{and} \quad
  h^{+}_{x,a}\ =\ \rho^{a-x}h^+_{x,x}.
\end{equation}
Then, defining $\tilde{h}^-$ and $h^-$ by
\begin{equation}
  \label{eq:relations-h+h-}
  \tilde{h}^-_{x,a}\ \coloneqq\ \rho^{-a}m_x h^+_{x,x} \quad \text{and} \quad
  h^-_{x,a}\ \coloneqq\ \frac{\rho^a}{m_x} \tilde{h}^+_{x,x},
\end{equation}
using the remaining symmetry between $M^+$ and $M^-$, it is
straightforward to check that $M^-$ also satisfies
Assumption~\ref{ass:recurrence}, with $\rho$ as its Perron eigenvalue,
and with $\tilde{h}^-$ and $h^-$ the associated left and right
eigenvectors.

We will assume a technical condition involving the quantities related
to our problem. Recall a nonnegative sequence $(u_x)_{x\geq 1}$
grows subexponentially if it satisfies~\eqref{eq:subexponential_growth}, which we can also write $u_x\leq \mathrm{e}^{o(x)}$.
This property is stable under several operations: term-wise sums and
products, partial sums and partial supremums in particular. Our
technical assumption is the following:
\begin{equation}
  \label{eq:technical-condition-appli}
    m_x + \sigma_{x,-}+\sigma_{x,+} + h^+_{x,x}+h^-_{x,x} +
    \frac{1}{h^+_{x,x}} + \frac{1}{h^-_{x,x}} \ \leq \ \mathrm{e}^{o(x)}.
\end{equation}
This condition can be interpreted as a sort of regularity condition
regarding both the strength of potential contacts along time for a
given individual, and the variations between infected individuals of
different types (we recall that $h_{x,x}^+$ grossly quantifies how
many descendants an infection at infection-age $x$ will have on the
long term). As a sanity check, we point out
that~\eqref{eq:technical-condition-appli} is automatically satisfied
in the trivial ``homogeneous'' case, when the degree sequences are
i.i.d., in which case the values $\sigma^\pm_x$ and $h^\pm_{x,x}$ do
not depend on $x$.

Note that, in the definition of $\mathcal{T}$, the number of stubs
assigned to the root vertex is given by a copy of $(D^\pm(x))_{x\geq
  1}$. Therefore, in order to apply our results, we need 
$\mathbb{E}[\langle Z^\pm(0), h^\pm\rangle]$ to be finite. Using the
relationships~\eqref{eq:relations-hxa-hxx}, this translates to
\begin{equation*}
  \mathbb{E}\big[\langle Z^\pm(0), h^\pm\rangle\big] \
  =\ \sum_{x\geq 1} m_x h^\pm_{x,1}
  =\ \sum_{x\geq 1} \rho^{1-x}m_x h^\pm_{x,x} \ < \infty,
\end{equation*}
which is satisfied as a consequence
of~\eqref{eq:technical-condition-appli}.
Note that because of~\eqref{eq:relations-h+h-}, this also implies that
\begin{equation}
  \label{eq:sum_tilde_h_finite}
  \sum_{x\geq 1} \tilde{h}^\pm_{x,x} \ =\ \rho^{-1} \sum_{x\geq 1} m_x
  h^{\mp}_{x,1} \ < \infty.
\end{equation}

With these assumptions, one can define the additive martingales
$(\rho^{-n}Z^\pm(n))_{n\geq 0}$ and their pointwise limits $W^{\pm}$.

Let us point out that because of the initial condition (the degree
sequences of the source vertex), $W^+$ and $W^-$ are in general not
independent, but because of the branching property, they are
independent conditional on $Z^\pm(0)$.

In order to state our result, we introduce a final notation in the
context of the infection graph $\mathcal{G}_N$. For $t\geq 0$, let us
define
$X^+(t)=(X^+_{x,a}(t))_{(x,a)\in\mathcal{X}} \in
\mathbb{N}^\mathcal{X}$, where for each $(x,a)\in \mathcal{X}$,
$X^+_{x,a}(t)$ counts the number of outstubs in $\mathcal{G}_N|_{t}$
of type $x$ and with age $a$, \textit{i.e.} those attached to a vertex at
directed distance $t-a+1$ from the root. These variables describe the
future (relative to time $t$) infectious contacts in the population
during the epidemic. We define $X^-_{x,a}(t)$ analogously.

\begin{theorem}\label{thm:application}
  Under the
  assumptions~\eqref{eq:moments},~\eqref{eq:Mplus_satisfies_I}
  and~\eqref{eq:technical-condition-appli} and using the
  notation above, fix $\eta \in (0,1/2)$ and define $t_N$ by
  \begin{equation*}
    t_N\ =\  \big\lfloor \eta\log_{\rho} N\big\rfloor.
  \end{equation*}
  Then the following limit holds in distribution in
  $\mathbb{B}_{h^+}\times \mathbb{B}_{h^-}$:
  \[
    \rho^{-t_N}\big(X^+(t_N),\ X^-(t_N)\big)\
    \overset{d}{\underset{N\to\infty}{\longrightarrow}} \
    \big(W^+\tilde h^{+}, W^-\tilde h^{-}\big).
  \]
\end{theorem}
We prove this as a corollary of the following coupling result, whose proof is deferred to the appendix.
We use $d_{\mathrm{TV}}$ to denote the total variation distance.

\begin{prop}\label{prop:coupling}
Assume the above assumptions hold, and let $t_N$ be defined as in
Theorem~\ref{thm:application}. Then,
\begin{equation*}
d_{\mathrm{TV}}\big(\mathcal{G}|_{t_N},\mathcal{T}|_{t_N}\big)\ \underset{N\to\infty}{\longrightarrow}\ 0.
\end{equation*}
\end{prop}

\begin{proof}[Proof of Theorem~\ref{thm:application}]
  By Theorem~\ref{thm:type_convergence}, the type distribution of
  $\rho^{-n} Z^{\pm}(t_N)$ converges in probability to
  $W^{\pm}\tilde h^{\pm}$ (the convergence is almost sure if $W^\pm=0$
  almost surely, and in $L^1$ otherwise). By
  Proposition~\ref{prop:coupling}, we can build versions of
  $X^\pm(t_N)$ (as $N$ varies) that are equal to $Z^{\pm}(t_N)$ with
  high probability, therefore the result follows.
\end{proof}

\begin{appendix}
\section{Proof of Proposition~\ref{prop:coupling}}

We now decompose the proof of Proposition~\ref{prop:coupling} in two
lemmas.

\begin{lem} \label{lem:proof_coupling_1} Using the notation of
  Section~\ref{sec:applications}, the following holds for all $t\geq 1$:
  \begin{multline}
    \label{eq:bound_TV_lemma}
      d_{\mathrm{TV}}\big(\mathcal{G}|_{t},\mathcal{T}|_{t}\big) \
      \leq\ \mathbb{P}\big(\abs{\mathcal{T}|_t} > N\big) \\
      +  \mathbb{E}\Big[1\wedge \bigg(\sum_{x\leq
        t}\frac{E_x(\mathcal{T}|_{t})}{m_x} \bigg(
        \Big(\frac{{K}_N^{x}}{N} - m_{x}\Big)_+
        +\frac{S_x(\mathcal{T}|_{t})}{N} (m_x + 1) \bigg)
        \bigg)\Big],
    \end{multline}
  where $\abs{\mathcal{T}|_t}$ denotes the number of vertices in
  $\mathcal{T}|_t$, $E_x(\mathcal{T}|_{t})$ the number of length-$x$
  edges in $\mathcal{T}|_t$, $S_x(\mathcal{T}|_{t})$ the sum of
  type-$x$ degrees, \textit{i.e.} the number of type-$x$ stubs and
  half-edges in $\mathcal{T}|_{t}$, and where the random variables
  $({K}_N^x)_{x\geq 1}$ are independent of $\mathcal{T}|_{t}$ and with
  distribution given by
  \begin{equation}\label{eq:K_N_x}
    K^x_N\ =\ \max\Big(\sum_{i=1}^N D^{+}_{i}(x),\ \sum_{i=1}^N D^{-}_{i}(x) \Big).
  \end{equation}
\end{lem}

\begin{proof}
Let us start with two important technical remarks before trying to
couple exactly $\mathcal{G}_N|_{t}$ and $\mathcal{T}|_{t}$. First
note that in these objects, by definition there are no edges with
length greater than $t$. Therefore when doing computations for fixed
$t$, we may fix an arbitrary cutoff value $T>t$ and couple the
modifications of the rooted decorated directed networks obtained by
removing all stubs with type $x\geq T$. It will turn out that the total
variation bound we get does not depend on $T$, so letting
$T\to\infty$, we will get the same total variation bound between the
unmodified networks $\mathcal{G}_N|_{t}$ and $\mathcal{T}|_{t}$.
Removing large stubs boils down to considering random variables
$(D^\pm(x))_{1\leq x\leq T}$ taking values in a discrete space,
which simplifies total variation computations.

Secondly, the exact computations that follow are much easier if we add
some combinatorial structure on the networks, namely an \emph{ordering
  of stubs and edges} in the following way. For each vertex $v$, and
type $x$, we order the finite set of type-$x$ outstubs and outgoing
edges attached to $v$ with a uniformly chosen total order; the same is
done for the set of type-$x$ instubs and incoming edges. Let us denote
by $\overline{\mathcal{G}}_N|_{t}$ and
$\overline{\mathcal{T}}|_{t}$ the \emph{ordered} rooted decorated
directed networks we obtain; for simplicity, let us call these
combinatorial objects \emph{configurations}. We will additionally
denote by $\overline{\mathcal{G}}_N|_{t,T}$ and
$\overline{\mathcal{T}}|_{t,T}$ the configurations with type-$x$
stubs removed when $x\geq T$.

For a fixed configuration $\mathcal{C}$ with $k$ vertices, in the
following let us fix an arbitrary labeling of vertices by integers
from $1$ to $k$, such that $1$ always denotes the root. We write
$d_i^\pm=(d_i^\pm(x))_{x\geq 1}$ for the in- or outdegrees of vertex
$i$ (counting type-$x$ edges and stubs). Let us write
$\mathscr{C}_{t,T}$ for the set of configurations $\mathcal{C}$ that
are without loops (\textit{i.e.} tree-like) and
(a) are compatible with our notion of restriction at distance $t$ in the
sense that $\mathcal{C}|_t = \mathcal{C}$, and (b) have no edges and
stubs with type or length $x\geq T$.
In other words, $\mathscr{C}_{t,T}$ is such that
\[
  \mathbb{P}\big(\overline{\mathcal{T}}|_{t,T} \in
  \mathscr{C}_{t,T}\big)\ =\ 1.
\]
These technical definitions finally allow for easy computations.
Indeed, for instance when a type-$x$ outstub attached to a vertex $u$
is explored in the definition of $\mathcal{T}$, the stubs attached to
the new vertex $v$ are distributed as $L^+_x$ and the rank of the
instub attached to $v$ is uniform among the $d_v^-(x)$ possibilities.
As a result, the probability to see a fixed degree distribution
$(d_v^\pm(y))_{1\leq y \leq T}$ and fixed ordering of the stubs of $v$ is
\[
  \frac{d_v^-(x)}{m_x}\mathbb{P}\big((D^\pm(y))_{1\leq y < T}=
  (d_v^\pm(y))_{1\leq y\leq T}\big) \times \frac{1}{d_v^-(x)} \ =\
  \frac{1}{m_x} p_T(d_v^\pm),
\]
where we write for conciseness
\[
  p_T(\,\cdot\,) \ \coloneqq\ \mathbb{P}\big((D^\pm(x))_{1\leq x < T}= (\,\cdot\,)_{1\leq x < T}\big).
\]
Exchanging the role of $+$ and $-$ reveals the same computation for
the exploration of instubs. As a result, using the branching structure
that defines $\mathcal{T}$, we simply get, for any
$\mathcal{C}\in \mathscr{C}_{t,T}$,
\begin{equation}\label{eq:prob_C_branching}
  \mathbb{P}\big( \overline{\mathcal{T}}|_{t,T} = \mathcal{C}\big)
  \ =\ p_T(d^\pm_{1})
  \prod_{\substack{e: u\rightarrow v \\ \text{or }u\leftarrow v}} \frac{1}{m_{x_e}}p_T(d^\pm_{v}),
\end{equation}
where the notation $e:u\to v$ means that we are considering edges
directed away from the root, going from $u$ to $v$, and
$e:u\leftarrow v$ means that we consider edges directed towards the
root, going from $v$ to $u$. In both cases $x_e$ denotes the length of
the edge.

The computation of
$\mathbb{P}(\overline{\mathcal{G}}_N|_{t,T} = \mathcal{C})$ is
slightly more complex, and relies mainly on the fact that the
probability that $k$ given matches are present in a uniform bipartite
matching with $K$ total matches is simply
\[
  \frac{1}{(K)_{k}}\ =\ \frac{1}{K (K-1) \cdots (K-k+1)} \ \geq\ \frac{1}{K^k}.
\]
Now recall that $\mathcal{G}_N$ is built from matching in- and
outstubs. Therefore, let us write $K^x_N$ for the total number of
matches of type-$x$ stubs needed to build $\mathcal{G}_N$,
\textit{i.e.} we define $K^x_N$ by~\eqref{eq:K_N_x}.
For a fixed configuration $\mathcal{C}$ with $k\leq N$ vertices
labeled arbitrarily from $1$ to $k$, let us denote by
$A_{\mathcal{C}}$ the event
\[
  A_{\mathcal{C}}\ \coloneqq\ \big\{ \forall 1\leq i\leq k, \ D^{\pm}_{i} = d^\pm_i \big\},
\]
where, as above, $(d^\pm_i)_{1\leq i\leq k}$ denote the degree
sequences of the vertices of $\mathcal{C}$. We claim that
\begin{equation}
  \label{eq:prob_C_graph}
  \mathbb{P}\big( \overline{\mathcal{G}}_N|_{t,T} = \mathcal{C} \big)\ =\
  (N-1)_{k-1} \,\cdot\, \prod_{i=1}^k p_T(d^\pm_i) \,\cdot\,
  \mathbb{E}\Big[\prod_{x=1}^t \frac{1}{(K_N^x)_{E_x}} \ \Big\vert\ A_{\mathcal{C}} \Big],
\end{equation}
where $E_x=E_x(\mathcal{C})$ denotes the number of length-$x$ edges in $\mathcal{C}$.
To obtain this, we argue that:
\begin{itemize}
\item The first factor counts the number of ways to label every vertex
  except from the root in configuration $\mathcal{C}$ with integers
  from $2$ to $N$.
\item The second factor is the probability to assign exactly the right
  degree sequences to these $k$ fixed vertices (recall that in the
  construction of $\mathcal{G}_N$, vertices are equipped with degree
  sequences in an i.i.d.\ way).
\item The third factor is the probability that to obtain the partial
  matchings given by the edges of $\mathcal{C}$ in each of the
  configuration models defining $\mathcal{G}_N$.
\end{itemize}
The key is that, when $\mathcal{C}\in \mathscr{C}_{t,T}$ we may
use~\eqref{eq:prob_C_branching},~\eqref{eq:prob_C_graph} and the fact that $k-1$ is the number
of edges to bound
\begin{align}\nonumber
  \mathbb{P}\big( \overline{\mathcal{G}}_N|_{t,T} = \mathcal{C} \big)\ 
  &\geq\ 
    \prod_{i=1}^k p_T(d^\pm_i) \,\cdot\,
    \mathbb{E}\Big[\prod_{e\in \mathcal{C}}
  \frac{N-k+1}{K_{N}^{x_e}} \ \Big\vert\ A_{\mathcal{C}} \Big], \\
  &=\ \mathbb{P}\big( \overline{\mathcal{T}}|_{t,T} = \mathcal{C} \big) \,\cdot\,
    \mathbb{E}\Big[\prod_{e\in \mathcal{C}}
  \frac{(N-k+1)m_{x_e}}{K_{N}^{x_e}} \ \Big\vert\ A_{\mathcal{C}} \Big],
    \label{eq:prob_C-G_compare}
\end{align}
where the product over $e\in \mathcal{C}$ is a slight abuse of
notation for the product over the edges of $\mathcal{C}$, that we
wrote before $e: i\to j$ or $e:j\to i$.

From now on we will use sums over $\mathcal{C}$ in $\mathscr{C}_{t,T}$
so to avoid confusion, we denote by $\abs{\mathcal{C}}$ the number of
vertices of a configuration. We exploit the following form of the
distance in total variation:
\begin{equation*}
    d_{\mathrm{TV}}\big(\overline{\mathcal{G}}_{N}|_{t,T},\overline{\mathcal{T}}|_{t,T}\big)\
    \leq \ \sum_{\mathcal{C}\in \mathscr{C}_{t,T}, \abs{\mathcal{C}}\leq N}
    \big(\mathbb{P}(\overline{\mathcal{T}}|_{t,T} = \mathcal{C}) -
    \mathbb{P}(\overline{\mathcal{G}}_{N}|_{t,T} = \mathcal{C})\big)_+ 
    + \mathbb{P}\big(\big|\overline{\mathcal{T}}|_{t,T}\big| > N \big).
\end{equation*}
As the last term does not depend on $T$ nor on the ordering of
$\mathcal{T}$, we recover the first term on the right-hand side
of~\eqref{eq:bound_TV_lemma}. We now deal with the sum.
Using \eqref{eq:prob_C-G_compare} and the inequality
$(1-\prod_i a_i)_+ \leq 1\wedge \big(\sum_i(1- a_i)_+)$ for any
family of nonnegative terms $a_i$ (this is readily checked by
induction), we obtain that
\begin{align}\nonumber
  \sum_{\mathcal{C}\in \mathscr{C}_{t,T}, \abs{\mathcal{C}}\leq N}
    \big(\mathbb{P}&(\overline{\mathcal{T}}|_{t,T} = \mathcal{C}) -
    \mathbb{P}(\overline{\mathcal{G}}_{N}|_{t,T} = \mathcal{C})\big)_+ \\ \nonumber
  &\quad \leq\ \sum_{\abs{\mathcal{C}}\leq N}
    \mathbb{P}\big(\overline{\mathcal{T}}|_{t,T} =
    \mathcal{C}\big)\mathbb{E}\Big[\Big(1-\prod_{e \in \mathcal{C}}\frac{N m_{x_e}(1-(|\mathcal{C}|-1)/N)}{K_{N}^{x_e}-\kappa^{x_e}_e}\Big)_+\, \Big\vert\, A_{\mathcal{C}}\Big]
  \\ \nonumber
  &\quad \leq\ \sum_{|\mathcal{C}|\leq N}
    \mathbb{P}\big(\overline{\mathcal{T}}|_{t,T} =
    \mathcal{C}\big)\mathbb{E}\Big[1\wedge \bigg(\sum_{e\in
    \mathcal{C}}\Big(1-\frac{N m_{x_e} }{K_N^{x_e}}\Big)_++
    \frac{|\mathcal{C}| - 1}{N} \bigg) \, \Big\vert\, A_{\mathcal{C}}\Big]
  \\ \label{eq:TV_bound}
  &\quad \leq\ \sum_{|\mathcal{C}|\leq N}
    \mathbb{P}\big(\overline{\mathcal{T}}|_{t,T} = \mathcal{C}\big)\,
    \mathbb{E}\Big[1\wedge \bigg(\sum_{x\leq t}E_x(\mathcal{C}) \Big(\Big(1-\frac{N m_{x}
    }{K_N^{x}}\Big)_+ + \frac{1}{N}\Big) \bigg) \, \Big\vert\, A_{\mathcal{C}}\Big],
\end{align}
where we recall that $E_x(\mathcal{C})$ denotes the number of edges of
length $x$ in $\mathcal{C}$. In order to remove the conditioning
in~\eqref{eq:TV_bound}, we bound $K_N^x$ using a copy of itself that
does not involve vertices in $\mathcal{C}$: on the event
$A_{\mathcal{C}}$, for all $x\leq t$,
\[
  K_N^x\ \leq\ S_x(\mathcal{C}) + \tilde{K}_N^x, \qquad \text{with
  }\tilde{K}_N^x\ \coloneqq\
  \bigg( \sum_{i=\abs{\mathcal{C}}+1}^{\abs{\mathcal{C}}+N}
  D^+_i(x)\bigg)\vee \bigg(\sum_{i=\abs{\mathcal{C}}+1}^{\abs{\mathcal{C}}+N}
  D^-_i(x)\bigg),
\]
where $S_x(\mathcal{C})$ is defined by
\[
  S_x(\mathcal{C}) \ \coloneqq\ \sum_{u\in \mathcal{C}}(d_u^+(x)+d_u^-(x)).
\]
Since for $a, b>0$, we have
$(1-1/(a+b))_+ \leq (a-1+b)_+\leq (a-1)_+ + b$, it
follows that
\[
  \Big(1-\frac{N m_{x} }{K_N^{x}}\Big)_+\ \leq\ \Big(\frac{\tilde{K}_N^{x}}{N
    m_{x}} - 1\Big)_+ + \frac{S_x(\mathcal{C})}{Nm_x}.
\]
Plugging this, we can bound~\eqref{eq:TV_bound} by
\begin{equation}
  \label{eq:bound_TV_2}
  \mathbb{E}\bigg[1\wedge \bigg(\sum_{x\leq
    t}E_x(\overline{\mathcal{T}}|_{t,T})\bigg(
  \Big(\frac{\tilde{K}_N^{x}}{N m_{x}} - 1\Big)_+ + \frac{S_x(\overline{\mathcal{T}}|_{t,T})}{N} \Big(\frac{1}{m_x} + 1 \Big)
  \bigg)\bigg],
\end{equation}
where $\mathcal{T}$ is independent of the variables
$(\tilde{K}_N^x)_{x\geq 1}$. Note that
$E_x(\overline{\mathcal{T}}|_{t,T})$ and
$S_x(\overline{\mathcal{T}}|_{t,T})$ do not depend on $T$ nor on the
ordering of edges and stubs, so $\overline{\mathcal{T}}|_{t,T}$ may be
simply replaced by $\mathcal{T}|_{t}$. We have shown a total variation
bound on the \emph{ordered} versions of $\mathcal{G}_N|_t$ and
$\mathcal{T}|_t$, and removing the orders may only decrease this
distance, so this concludes the proof of the lemma.
\end{proof}

\begin{lem} \label{lem:proof_coupling_2}
  For all $x,N\geq 1$, 
  \begin{multline}
  \label{eq:bound_S_x_2}
    \mathbb{E}[S_x(\mathcal{T}|_{t_N})]\ \leq\ \Big(\sup_{y\leq
    t_N}\big(h^-_{y,y}\wedge h^+_{y,y}\big)^{-1} \Big)
     \frac{\rho^{t_N+2}}{\rho-1}
    \Big(\tilde{h}^+_{x,1} +
      \tilde{h}^-_{x,1} + m_x \sum_{x\geq 1}(\tilde{h}^+_{x,x} + \tilde{h}^-_{x,x}) \\
   + \frac{\sigma_{x,+}}{\rho}\sum_{y\leq t_N}
      h_{y,y}^- \sigma_{y,+} + \frac{\sigma_{x,-}}{\rho}\sum_{y\leq t_N}
      h_{y,y}^+ \sigma_{y,-} \Big),
\end{multline}
and
\begin{equation} \label{eq:bound_E_x}
  \mathbb{E}\Big[\frac{E_x(\mathcal{T}|_{t_N})}{m_x}\Big]\ \leq\
  \Big(\sup_{y\leq t_N}\big(h^-_{y,y}\wedge h^+_{y,y}\big)^{-1} \Big)
  \frac{\rho^{t_N+2}}{\rho-1} (h^-_{x,x}+h^+_{x,x}).
\end{equation}
Hence under Assumption~\eqref{eq:technical-condition-appli}, the subexponential bound follows:
\begin{equation}
  \label{eq:bounds_E_and_S}
  \sum_{x\leq t_N}\rho^{-t_N}
  \Big(\mathbb{E}\Big[\frac{E_x(\mathcal{T}|_{t_N})}{m_x}\Big]
  + \mathbb{E}\big[S_x(\mathcal{T}|_{t_N})\big] \Big) \ \leq\
  \mathrm{e}^{o(t_N)}.
\end{equation}
\end{lem}

\begin{proof}
Let us come back to the branching process to count the stubs in
$\mathcal{T}|_{t_N}$. Let us focus on the $x$-instubs. Those that are
connected to a vertex $v$ that may reach the root (\textit{i.e.}
such that $\vec{d}(v,1)\leq t_N$) are simply counted by
\[
  \sum_{t\leq t_N}Z^-_{x,1}(t).
\]
To count those that are connected to a vertex $v$ reachable from the
root, we first count those vertices and then use the branching
property: given that the edge pointing to $v$ is of type $y$, the mean
number of $x$-instubs is by definition given by
$1/m_y \mathbb{E}[D_y^-D_x^-]$. A similar reasoning holds for
$x$-outstubs, so after some minor factoring, we get
\begin{multline}
  \label{eq:bound_S_x}
    \mathbb{E}\big[S_x(\mathcal{T}|_{t_N})\big]\
     =\ \sum_{t\leq t_N}\mathbb{E}\big[Z^+_{x,1}(t) + Z^-_{x,1}(t)\big] \\
     + \sum_{t\leq t_N} \sum_{y\leq t_N} \frac{1}{m_y}\Big(
      \mathbb{E}\big[Z^+_{y,y}(t)\big]\mathbb{E}\big[D_y^-D_x^-] + \mathbb{E}\big[Z^-_{y,y}(t)\big]\mathbb{E}\big[D_y^+D_x^+] \Big).
  \end{multline}
To bound the $Z^\pm$ terms, we may use the properties of the
matrices $M^\pm$. Indeed, recall that
$\mathbb{E}[Z^\pm_{x,a}(0)]=m_x\mathds{1}_{a=1}$ for all
$(x,a)\in \mathcal{X}$. By the age structure of the branching
process $Z^+$, we have for all $t\leq t_N$ and $(y,b)\in \mathcal{X}$,
\begin{align*}
  \mathbb{E} \big[Z^+_{y,b}(t)\big]\ &=\ \Big((m_x\mathds{1}_{\{a=1\}})_{(x,a)\in
    \mathcal{X}} \times (M^+)^t\Big)_{y,b}\\
    &=\ \Big((m_x\mathds{1}_{\{a=1,x\leq
  t_N\}})_{(x,a)\in
    \mathcal{X}} \times (M^+)^t\Big)_{y,b},
\end{align*}
where $\times$ denotes matrix multiplication.
Note that
\[
  m_x\mathds{1}_{\{a=1,x\leq t_N\}}\ \leq\ \alpha_N
  \tilde{h}^+_{x,a}\wedge \tilde{h}^-_{x,a}, \qquad (x,a)\in \mathcal{X},
\]
with
\[
  \alpha_N^{-1}\ =\ \inf_{x\leq t_N} \frac{\tilde{h}^+_{x,1}\wedge\tilde{h}^-_{x,1}}{m_x}\ =\
  \rho^{-1} \inf_{x\leq t_N} h^-_{x,x} \wedge h^+_{x,x},
\]
where we used the relations~\eqref{eq:relations-h+h-}. Therefore, we get
$\mathbb{E} [Z^+_{y,b}(t)] \leq \alpha_N \rho^t \tilde{h}^+_{y,b}$
(and a similar bound for $Z^-$), so plugging this
into~\eqref{eq:bound_S_x}, we get
\[
  \mathbb{E}[S_x(\mathcal{T}|_{t_N})]\ \leq\ \alpha_N
  \frac{\rho^{t_N+1}}{\rho-1} \Big(\tilde{h}^+_{x,1} +
  \tilde{h}^-_{x,1} + \sum_{y\leq t_N} \frac{1}{m_y}\Big(
  \tilde{h}^+_{y,y}\mathbb{E}[D_y^-D_x^-] +
  \tilde{h}^-_{y,y}\mathbb{E}[D_y^+D_x^+] \Big)\Big).
\]
We now use the straightforward bound
\[
  \mathbb{E}[D^\pm_xD^\pm_y]\ \leq\ m_xm_y + \sigma_{x,\pm}\sigma_{y,\pm},
\]
where we recall that $\sigma^2_{x,\pm}$ denotes the variance of
$D^\pm_x$, to get~\eqref{eq:bound_S_x_2} after some simplification.
The inequality~\eqref{eq:bound_E_x} is obtained by similar
computations which we leave to the reader. Note that $\sum_{x\geq
  1}\tilde{h}^\pm_{x,x}$ is finite by~\eqref{eq:sum_tilde_h_finite}, and
that by~\eqref{eq:relations-h+h-}, we can rewrite $\tilde{h}^\pm_{x,1}
= \rho^{-1} m_x h^\mp_{x,x}$.
Finally~\eqref{eq:bounds_E_and_S} is obtained by noting that our
bounds divided by $\rho^{t_N}$ are obtained from sequences growing
subexponentially -- by the
assumption~\eqref{eq:technical-condition-appli} -- only through
operations preserving the property of subexponential growth (sums,
products, partial sums and supremums).
\end{proof}

We can now combine Lemmas~\ref{lem:proof_coupling_1}
and~\ref{lem:proof_coupling_2} to show Proposition~\ref{prop:coupling}.

\begin{proof}[Proof of Proposition~\ref{prop:coupling}]
We now take $t$ to equal $t_N$ defined as in
Theorem~\ref{thm:application}, and show
that~\eqref{eq:bound_TV_lemma} tends to $0$ as $N\to\infty$. First let
us tackle the second term on the right-hand side. Because
of the bound by $1$ inside the expectation, it suffices to show that
the random variable
\begin{equation}
  \label{eq:rv_goes_to_0}
    \sum_{x\leq
      t}\frac{E_x(\mathcal{T}|_{t})}{m_x} \bigg(
      \Big(\frac{{K}_N^{x}}{N} - m_{x}\Big)_+
      +\frac{S_x(\mathcal{T}|_{t})}{N} (m_x + 1) \bigg)
\end{equation}
goes to $0$ in probability.

For the first part, recalling that $\sigma_{x,\pm}^2$ is the
variance of $D^{\pm}_x$, we can crudely bound
\begin{align*}
  \mathbb{E}\Big[\Big(\frac{\tilde{K}_N^{x}}{N}-m_x\Big)_+\Big]\
  &\leq\ \mathbb{E}\Big[\Big|\sum_{i=1}^N \frac{D^+_i(x)}{N} - m_x\Big| + \Big|\sum_{i=1}^N \frac{D^-_i(x)}{N} - m_x\Big|\Big]\\
  &\leq\ \frac{1}{\sqrt{N}} (\sigma_{x,+}+\sigma_{x,-}).
\end{align*}
Putting the last display together with~\eqref{eq:bound_E_x} yields
\begin{equation*}
    \mathbb{E} \bigg[\sum_{x\leq t_N}\frac{E_x(\mathcal{T}|_{t_N})}{m_x}
  \Big(\frac{\tilde{K}_N^{x}}{N} - m_x\Big)_+\bigg] \
   =\ O\bigg(\frac{\rho^{t_N}}{\sqrt{N}} \sum_{x\leq t_N}
(\sigma_{x,+}+\sigma_{x,-})\frac{h^+_{x,x}+h^-_{x,x}}{\inf_{y\leq t_N}
    h^-_{y,y}\wedge h^+_{y,y}} \bigg),
\end{equation*}
which tends to $0$ by~\eqref{eq:technical-condition-appli}
because $\rho^{t_N}=O(N^{\eta})$, where $\eta<1/2$. For the second part
of~\eqref{eq:bound_TV_2}, we first take square roots before taking the expectation and applying the Cauchy--Schwarz inequality:
\begin{align*}
  \mathbb{E}\bigg[ \bigg(\sum_{x\leq
    t_N}&\frac{E_x(\mathcal{T}|_{t_N})}{m_x}\frac{S_x(\mathcal{T}|_{t_N})}{N}
    (m_x + 1) \bigg)^{\frac{1}{2}} \bigg] \\
  &\leq\ \frac{1}{\sqrt{N}} \mathbb{E}\bigg[ 1 + \sum_{x\leq
    t_N}\bigg(\frac{E_x(\mathcal{T}|_{t_N})}{m_x}S_x(\mathcal{T}|_{t_N})
    (m_x + 1) \bigg)^{\frac{1}{2}} \bigg] \\
  &\leq\ \frac{1}{\sqrt{N}} \bigg(1 + \sum_{x\leq
    t_N}\mathbb{E}\Big[\frac{E_x(\mathcal{T}|_{t_N})}{m_x}\Big]^{\frac{1}{2}}
    \mathbb{E}[S_x(\mathcal{T}|_{t_N})]^{\frac{1}{2}}
    (m_x + 1)^{\frac{1}{2}}\bigg) \\
  &\leq\ \frac{1}{\sqrt{N}} \bigg(1 + \bigg(\sum_{x\leq
    t_N}\mathbb{E}\Big[\frac{E_x(\mathcal{T}|_{t_N})}{m_x}\Big]
    \bigg)^{\frac{1}{2}} \bigg(\sum_{x\leq t_N}
    \mathbb{E}[S_x(\mathcal{T}|_{t_N}) (m_x +
    1)]\bigg)^{\frac{1}{2}} \bigg).
\end{align*}
Using~\eqref{eq:bounds_E_and_S} and the fact that
$\rho^{t_N}=O(N^\eta)$ with $\eta<1/2$, this tends to $0$ as
$N\to \infty$. Therefore, we have shown that the random
variable~\eqref{eq:rv_goes_to_0} converges to $0$ in probability, which implies that the second term on the right-hand side
of~\eqref{eq:bound_TV_lemma} vanishes. Bounding
$\abs{\mathcal{T}|_{t_N}} - 1$ by the sum of the
$E_x(\mathcal{T}|_{t_N})$, similar computations easily yield
$\mathbb{P}(|\mathcal{T}|_{t_N}|>N)\to 0$, so the proof is complete.
\end{proof}
\end{appendix}

\paragraph{Acknowledgements.} We thank Maxime Ligonnière for insightful discussions regarding the (local or not) extinction of the branching process, in particular concerning the positive recurrent example where $\{W=0\}=\mathrm{LocExt}\neq \mathrm{Ext}$.

We acknowledge funding through the ANR project GARP (ANR-24-CE40-7154).
Mathilde André acknowledges funding from the Center for Interdisciplinary Research in Biology (CIRB), Collège de France and the École Polytechnique.

\phantomsection
\addcontentsline{toc}{section}{References}


\end{document}